\documentclass[11pt,reqno]{amsart}
\usepackage{amssymb,amsmath}
\usepackage{amsthm}
\usepackage{color}
\usepackage{hyperref}
\usepackage{color}
\usepackage{graphicx}

\newcommand{\R}{\mathbb{R}}
\newcommand{\N}{{\mathbb N}}

\newcommand{\Z}{\mathbb Z}

\newcommand{\ds}{\displaystyle}
\newcommand{\sn}{{\mbox{SN}}}
\newcommand{\cn}{{\mbox{CN}}}
\newcommand{\dn}{{\mbox{DN}}}

\numberwithin{equation}{section}

\newtheorem{theorem}{Theorem}[section]
\newtheorem{proposition}[theorem]{Proposition}
\newtheorem{remark}[theorem]{Remark}
\newtheorem{lemma}[theorem]{Lemma}
\newtheorem{corollary}[theorem]{Corollary}
\newtheorem{definition}[theorem]{Definition}

\begin{document}
\vglue-1cm \hskip1cm
\title[Stability of periodic waves for the regularized Schamel equation]{Orbital stability of periodic traveling-wave solutions for the regularized Schamel equation}

%\author[]{}
%\address{}
%\email{}
%\thanks{}
%\begin{abstract}
%\end{abstract}

%\maketitle

\begin{center}

%%%%%%%%%%%%%%%%%%%%%%%%%%%%%%%%%%%%%%%%%%%%%%%%%%%%%%%%%%%%%%%%%%%%%%%%%%%%%%%%%%%%%%%%%%%%%%%%



\maketitle

{\bf Thiago Pinguello de Andrade}

{Departamento de Matem{\'a}tica - Universidade Tecnol{\'o}gica Federal do Paran{\'a}  \\
Av. Alberto Carazzai, 1640, CEP 86300-000, Corn{\'e}lio Proc{\'o}pio, PR, Brasil.}\\
{thiagoandrade@utfpr.edu.br}

\vspace{3mm}

{\bf Ademir Pastor}

{ IMECC-UNICAMP\\
Rua S\'ergio Buarque de Holanda, 651, CEP 13083-859, Campinas, SP,
Brazil.  } \\{ apastor@ime.unicamp.br}

\end{center}

\begin{abstract}
In this work we study the orbital stability of periodic traveling-wave solutions for dispersive models. The study of traveling waves started in the mid-18th century when John S. Russel established that the flow of water waves in  a shallow channel has constant evolution. In recent years, the general strategy to obtain orbital stability consists in proving that the traveling wave in question minimizes a conserved functional restricted to a certain manifold. Although our method can be applied to other models, we deal  with the regularized Schamel equation, which contains a fractional nonlinear term. We obtain a smooth curve of periodic traveling-wave solutions depending on the Jacobian elliptic functions and prove that such solutions are orbitally stable in the energy space.
In our context, instead of minimizing the augmented Hamiltonian in the natural codimension two manifold, we minimize it in a ``new'' manifold, which is suitable to our purposes.
 
\end{abstract}

\section{Introduction}

\indent This paper sheds new contributions in the sense of obtaining orbital stability of periodic traveling waves for nonlinear  dispersive models. Although, we pay particular attention to the regularized Schamel equation, 
\begin{equation}
u_t-u_{xxt}+(u+|u|^{3/2})_x=0,
\label{bbm}
\end{equation}
such ideas can be applied to a large class of dispersive models (see \cite{acn} and \cite{depas}).

The Korteweg-de Vries (KdV) equation,
\begin{equation}\label{kdvint}
u_t+u_{xxx}+\frac{1}{2}(u^2)_x=0,
\end{equation}
the regularized long-wave  or Benjamin-Bona-Mahony (BBM) equation
\begin{equation}\label{bbmint}
u_t-u_{xxt}+(u+\frac{1}{2}u^2)_x=0,
\end{equation}
and their various modifications are widely used models describing the propagation of nonlinear waves. Originally, \eqref{bbmint} was derived by Benjamin, Bona, and Mahony in \cite{bbm} as an alternative model to the KdV equation for small-amplitude, long wavelength surface water waves. 

For one hand, \eqref{bbm} can  be viewed as a regularized version of the Schamel equation
\begin{equation}\label{schamel}
u_t+u_{xxx}+(u+|u|^{3/2})_x=0,
\end{equation}
in much the same way that the BBM equation can be regarded as a regularized version of the KdV equation. Equation \eqref{schamel} was derived by Schamel \cite{sch}, \cite{sch1} as a model to describe the propagation of weakly nonlinear ion acoustic waves which are modified by the presence of trapped electrons. 

Both KdV and BBM equations also model the one dimensional waves in a cold plasma (see e.g., \cite{brosl}), with the difference that BBM equation describes much better the behavior of very short waves (see also \cite{tran}). The approach in \cite{brosl} is based on the use of approximate Hamiltonians. So, depending on the region of validity, a new equation for nonlinear ion waves is obtained. In a general setting, the canonical equations for one-dimensional waves, in dimensionless form, read as
\begin{equation}\label{appham}
\begin{cases}
 u_t-u_{xxt}=-\partial_x \dfrac{\delta H_a}{\delta v},\\
 v_t-v_{xxt}=-\partial_x \dfrac{\delta H_a}{\delta u},
\end{cases}
\end{equation}
where $H_a=H_a(u,v)$ is the approximate Hamiltonian and $u$ and $v$ are related to the ion mass density and ion mass velocity, respectively. If we take the Hamiltonian to be
$$
H_a=\frac{1}{2}\int \left( u^2+v^2+\frac{4}{5}\mbox{sgn}(u)|u|^{5/2} \right)dx
$$
we see that \eqref{appham} reduces to
\begin{equation}\label{appham1}
\begin{cases}
u_t-u_{xxt}=-\partial_x (v),\\\\
v_t-v_{xxt}=-\partial_x (u+|u|^{3/2}).
\end{cases}
\end{equation}
By putting $u+v=w$, $u-v=z$, adding the equations in \eqref{appham1}, and neglecting $z$, we find (up to constants) the BBM-like approximation for unidirectional waves as in \eqref{bbm} (in the variable $w$).

\begin{remark}
Although we will not discuss here, the region of validity of $H_a$ can be addressed as in \cite{brovati}. It is to be observed that the approximate Hamiltonian may be corrected to
$$
H_a=\frac{1}{2}\int \left( u^2+v^2+uv^2+\frac{4}{5}\mbox{sgn}(u)|u|^{5/2} \right)dx.
$$
By following the above steps, we obtain the equation
\begin{equation}\label{eqnonleq}
w_t-w_{xxt}=-\partial_x(w+aw^2+b|w|^{3/2}),
\end{equation}
where $a$ and $b$ are real constants. The nonlinearity appearing in \eqref{eqnonleq} is in agreement with \cite{sch1} (see also \cite{gill} and references therein for more recent results in this direction), where the author observed that in some physical situations, the nonlinearity in \eqref{schamel} should be corrected to $(u^2+|u|^{3/2})_x$. Following the ideas of our work, one can also study the existence and orbital stability of periodic traveling waves to \eqref{eqnonleq}.
\end{remark} 

Equations \eqref{bbm} and \eqref{schamel} as less studied than \eqref{bbmint} and \eqref{kdvint}, mainly because the fractional power in the nonlinear part brings a lot of difficulties which, in several aspects, cannot be handle with standard techniques.  We point out, however, that the spectral stability of periodic traveling-wave solutions for \eqref{schamel} was studied  in \cite{bjk}.

From the mathematical point of view, the generalized BBM equation
\begin{equation}\label{genbbm}
u_t-u_{xxt}+(f(u))_x=0
\end{equation}
has become a major topic of study in recent years and much effort has been expended on various aspects of \eqref{genbbm}. The issues include the initial-value (initial-boundary-value) problem, existence and stability of solitary and periodic traveling waves and global behavior of solutions. Thus, \eqref{bbm} can also be viewed as \eqref{genbbm} with $f(u)=u+|u|^{3/2}$. In this context, \eqref{bbm} has appeared, for instance, in \cite{ss} where the authors study the initial-value problem in the usual $L^2$-based Sobolev spaces and the nonlinear stability of solitary traveling waves (see also \cite{bmr}, \cite{johnson2}, \cite{lin} and references therein).

Our main goal in this work is to establish the existence and orbital stability of an explicit family of periodic traveling-wave solutions associated with \eqref{bbm}. The traveling waves we are interested in are of the form  $u(x,t)=\phi(x-ct)$, where $\phi$ is a periodic function of its argument and $c>1$ is a real parameter representing the wave speed. By replacing this form of $u$ in \eqref{bbm}, one sees that  $\phi$ must solve the nonlinear ODE
\begin{equation}
-c\phi ''+(c-1)\phi-\phi^{3/2}+A=0,
\label{el}
\end{equation}
where $A$ is an integration constant.

The constant $A$ in \eqref{el} plays a crucial role in the theory of nonlinear stability. Indeed, assume we are dealing with an invariant by translation Hamiltonian nonlinear evolution equation of the form $u_t=J\nabla E(u)$, where $E$ is the energy functional. Suppose the associated periodic traveling waves satisfy a conservative equation like
\begin{equation}\label{consedo}
-\phi''+h(\phi,c,A)=0,
\end{equation}
for some smooth function $h$. Here $c$ represents the wave speed and $A$ is an arbitrary constant. On one hand, when $A=0$ the, by now, classical stability theories \cite{be}, \cite{Grillakis}, \cite{weinstein1}, \cite{weinstein2} pass to showing that $\phi$ is a local minimum of $E$ restrict to a suitable manifold depending on the conserved quantity originated by translation invariance, say, $Q$. At this point, the coercivity of the functional $\widetilde{\mathcal{F}}=E+cQ$ develops a fundamental role and many works concerning the stability of periodic waves have appeared in the literature (see e.g., \cite{angulo7}-\cite{angulo1}, \cite{fp}, \cite{hik} \cite{NP1}-\cite{Neves1} and references therein). On the other hand, the situation  when $A\neq0$ is a little bit more delicate; the minimization of $\widetilde{\mathcal{F}}$ is not enough to produce the desired results  and, in general, we need to add an extra conserved quantity to $\widetilde{\mathcal{F}}$. Thus, it is reasonable to work with  a functional having the form $\mathcal{F}=E+cQ+AV,$ where $V$ is another conservation law. The quantity $V$ in general  does not come from an invariance of the equation. As a consequence, the theories mentioned above can not be directly applied  and this forces us to revisit its core  in order to cover the nonlinear stability in such situations (see also \cite{angulo4} and \cite{angulo1}).

Although in a different way, the case $A\neq0$ was addressed,  for instance  in \cite{johnson1}, \cite{johnson2} and \cite{natali2}, where the orbital stability of a three- or two-parameter family of periodic traveling waves associated with KdV-type and generalized BBM equations were established (see also \cite{angulo1}). In these works, the authors do not use the explicit form of the waves and prove the existence of local  families of traveling waves by using the standard theory of ODE's. Their method has the advantage that it avoids a lot of hard calculations which appear when explicit solutions are studied. However, as we will see below, our approach has the advantage that we can prove the needed spectral properties in a very simple way.

Let us now turn attention to the main steps of our constructions. As we already mentioned, our aim  consists in making some changes in the classical theories  developed in \cite{be}, \cite{Grillakis}, and \cite{weinstein1},  in order to deal with explicit periodic solutions of \eqref{el} obtained when $A\neq 0$.  With this in mind, we prove that \eqref{el} has a solution of the form
$$
\phi(\xi)=(\alpha+\beta \cn^2(\gamma \xi, k))^2,
$$
where $\cn$ denotes the Cnoidal Jacobi elliptic function, $k\in(0,1)$ represents the elliptic modulus and $\alpha, \beta$, and $\gamma$ are suitable constants depending, a priori, on $c$ and $A$. After some algebraic manipulations, one can write all parameters in terms of $k$, so that a smooth curve of explicit $L$-periodic solutions $k\in J\to\phi_k$ can be obtained.   In particular, the parameters $c$ and $A$ can also be written as functions of $k$.

The conserved quantities appearing here are
\begin{equation}
E(u)=\frac{1}{2}\int_0^L \left(u_x^2-\frac{4}{5}\mbox{sgn}(u)|u|^{\frac{5}{2}}\right)dx, \qquad Q(u)=\frac{1}{2} \int_0^L (u^2+u_x^2)dx,
\label{quantitiesconserved}
\end{equation}
and
\begin{equation}\label{quantitiesconserved1}
 V(u)=\int_0^Ludx.
\end{equation}
Note that $E$ is a smooth functional in any region that does not contain the origin of   $H^1_{per}([0,L])$.

According to our discussion above and taking into account \eqref{el}, the functional $\mathcal{F}$ becomes 
\begin{equation}
F_{(c,A)}(u):=E(u)+(c-1)Q(u)+AV(u),
\label{quantidadeconservadaF}
\end{equation}
and we will need to minimize $F_{(c,A)}$ constrained to a suitable manifold  $\Sigma_k$, defined below. To this end, spectral properties of the linearized operator
\begin{equation}
\mathcal{L}_{(c,A)}:=F_{(c,A)}''(\phi)=-c\frac{d^2}{dx^2}+(c-1)-\frac{3}{2}\phi^{\frac{ 1}{2}}
\label{operadorlinearizado}
\end{equation} 
will be necessary. Since we can see $c$ and $A$ as functions of $k$, we can write $F_{(c(k),A(k))}=:F_k$ and $\mathcal{L}_{(c(k),A(k))}=:\mathcal{L}_k$.  We obtain the spectral properties we need, taking into account that the periodic eigenvalue problem associated with the operator $(1/c)\mathcal{L}_k$ can be reduced to an eigenvalue problem associated with a Lam\'e-type equation, namely,
\begin{equation*}
\left\{\begin{array}{l}
\ds\Lambda''(x)+\left[h-5\cdot 6\cdot k^2 \sn^2\left(x,k\right)\right]\Lambda(x)=0,\\
\Lambda(0)=\Lambda(2K(k)), \quad \Lambda'(0)=\Lambda'( 2K(k)).
\end{array}\right.
\end{equation*}

To summarize all the constructions we need, we will prove that the following four properties hold.

\begin{itemize}
\item[($P_0$)] There exists a nontrivial smooth curve of $L$-periodic solutions of (\ref{el}), $ k\in J  \mapsto \phi_k:=\phi_{(c(k),A(k))} \in H^2_{per}([0,L])$, where $J\subset(0,1)$ is an interval to be determined later.\\
\item[($P_1$)] The linearized operator $\mathcal{L}_k:=\mathcal{L}_{(c(k),A(k))}$ has a unique negative  eigenvalue, which is simple.\\
\item[($P_2$)] Zero is a simple eigenvalue of $\mathcal{L}_k$ with associated  eigenfunction $\phi_k'$.\\
\item[($P_3$)] The quantity $\Phi$ defined by $\Phi:=\left\langle\mathcal{ L}_k\left(\frac{\partial \phi_k}{\partial k}\right),\frac{\partial \phi_k}{\partial k}\right\rangle$
 is negative.
\end{itemize}

Once property $(P_0)$ has been proved, our strategy is to show that properties $(P_1)-(P_3)$ combine to establish the orbital stability of that traveling waves (see Section \ref{section5}). Although our theory demands computations involving the explicit traveling waves, it is simpler to be applied once it does not demand the existence of a two-parameter family of traveling waves.

The paper is organized as follows. In  Section \ref{section2}, we prove a global well-posedness result for the Cauchy problem associated with \eqref{bbm}. Since we are not interested in optimal regularity, we only prove a global well-posedness result in the energy space $H^1_{per}([0,L])$, which is sufficient to our purpose. In Section \ref{section3}, we show the existence of an explicit curve of traveling waves, which  establishes property $(P_0)$. In Section \ref{section4}, we prove the operator $\mathcal{L}_k$ has a unique negative and simple eigenvalue and that $0$ is a simple eigenvalue, by showing thus properties $(P_1)$ and $(P_2)$. The strategy to do this, is to reduce the periodic eigenvalue problem associated with $\mathcal{L}_k$ to another problem involving the so-called Lam\'e equation.  Finally, in last section, Section \ref{section5}, we prove property $(P_3)$ and show how such a condition implies the orbital stability of the traveling waves presented in $(P_0)$. \\

\noindent {\bf Notation.} For $s\in\mathbb{R}$, the Sobolev space
$H_{per}^{s}=H_{per}^{s}([0,L])$ is the set of all periodic distributions
such that $||f||_{H^s_{per}}^2:=
L\sum_{k=-\infty}^{+\infty}(1+|k|^2)^s|\widehat{f}(k)|^2 <\infty, $ where
$\widehat{f}$ is the (periodic) Fourier transform of $f$. For $s=0$, $H_{per}^{s}([0,L])$ is isometric to $L_{per}^{2}([0,L])$. The norm and inner product  in $L^2_{per}$ will be denoted by $\|\cdot\|$ and $(\cdot,\cdot)_{L^2_{per}}$, respectively. By $\langle\cdot,\cdot\rangle$ we mean the duality pairing $H^1_{per}$-$H^{-1}_{per}$.  The symbols $\sn(\cdot,k)$,
$\dn(\cdot,k)$, and $\cn(\cdot,k)$ represent the Jacobi elliptic functions of
\emph{snoidal}, \emph{dnoidal}, and \emph{cnoidal} type, respectively. Recall
that $\sn(\cdot,k)$ is an odd function, while $\dn(\cdot,k)$, and
$\cn(\cdot,k)$ are even functions. For $k\in(0,1)$, $K(k)$ and $E(k)$ will
denote the complete elliptic integrals of the first and second type,
respectively (see e.g., \cite{friedman}).

\section{Global well-posedness}\label{section2}

\indent\indent
In this section we establish local and global well-posedness results for the periodic Cauchy  problem associated with \eqref{bbm}, namely,
\begin{equation}
\left\{\begin{array}{l}
u_t-u_{xxt}+(u+|u|^{\frac{3}{2}})_x=0, \quad x\in[0,L],\;t>0,\\
u(x,0)=u_0(x).
\end{array}
\right.
\label{bcbbm4}
\end{equation} 

 To simplify the notation, in what follows in this section, we set $L = 2\pi$. Our strategy to prove the local well-posedness of \eqref{bcbbm4}  consists in using the smoothness properties of the operator $K$, defined via its Fourier transform as
$$
\widehat{Ku}(k)=\frac{-ik}{1+k^2}\widehat{u}(k), \qquad k\in\Z,
$$
and to apply the contraction mapping principle to the equivalent integral equation, which is given by
\begin{equation}
u(x,t)=u_0(x)+\int_0^t K* (u+|u|^{\frac{3}{2}})(x,\tau)d\tau.
\label{bcbbm3}
\end{equation}

Our local well-posedness result reads as follows.

\begin{proposition}[Local well-posedness]\label{lowellp}
For each $u_0\in H^1_{per}([0,2\pi])$, there exist $T>0$ and a unique solution $u$ of \eqref{bcbbm3} such that $u\in C([0,T];H^1_{per}([0,2\pi]))$. Moreover, for every $T'\in(0,T)$, there exists a neighborhood $W$ of $u_0$ in $H^1_{per}([0,2\pi])$ such that the map data-solution 
$$\begin{array}{rl}
G: & W \longrightarrow C([0,T'];H^1_{per}([0,2\pi]))\\
& v_0 \longmapsto G(v_0) = v
\end{array}$$
is continuous.
\label{bcbbm5}
\end{proposition}
\begin{proof} Let $T>0$ be a fixed time, to be chosen later, and define the space
$$X_T:=C([0,T];H^1_{per}([0,2\pi])),$$
endowed with the norm
 $\ds\|u\|_{X_T}=\sup_{t\in [0,T]}\|u(t)\|_{H^1_{per}}$.
Also, define the operator $\mathcal{A}$ by 
$$\left\{\begin{array}{lcll}
\mathcal{A}:& X_T & \longrightarrow & X_T\\
& \ds u& \ds \longmapsto &\ds \mathcal{A}(u)= u_0 +\int_0^t K * (u+|u|^{\frac{3}{2}})d\tau.\\
\end{array}
\right.
$$

\noindent \textbf{Local Existence:} In order to apply the contraction mapping principle, we first prove there are $r_0>0$ and $T>0$ such that  $\mathcal{A}:B_{r_0}\longrightarrow B_{r_0}$ is a contraction. Here $B_{r_0}$ denotes the closed ball in $X_T$ centered at the origin with radius $r_0$.  Note that,
if $w\in{H^1_{per}([0,2\pi])}$ then  
\begin{equation}
\|K*w\|_{H^1_{per}}^2=\ds 2\pi \sum_{-\infty}^{+\infty} (1+k^2)\frac{k^2}{(1+k^2)^2}|\widehat{w}(k)|^2 \leq 2\pi \sum_{-\infty}^{+\infty} |\widehat{w}(k)|^2=\|w\|_{L^2_{per}}^2 
\label{bcbbm7}
\end{equation}
and
\begin{equation}
\|K*w\|_{H^1_{per}}^2=\ds 2\pi \sum_{-\infty}^{+\infty} (1+k^2)\frac{k^2}{(1+k^2)^2}|\widehat{w}(k)|^2 \leq 2\pi \sum_{-\infty}^{+\infty}(1+k^2) |\widehat{w}(k)|^2=\|w\|_{H^1_{per}}^2.
\label{bcbbm7.1}
\end{equation}
In addition, $|w|^{\frac{3}{2}}\in H^1_{per}([0,2\pi])$ and there exists $C_0>0$ such that
\begin{equation}
\||w|^{\frac{3}{2}}\|^2_{H^1_{per}} =  \||w|^{\frac{3}{2}}\|^2_{L^2_{per}}+\frac{9}{4}\||w|^\frac{1}{2}w'\|^2_{L^2_{per}}\leq C_0 \|w\|^3_{H^1_{per}},
\label{bcbbm8}
\end{equation}
where we used  Sobolev's embedding in the last inequality.
Therefore,  it follows from  \eqref{bcbbm7.1} and \eqref{bcbbm8} that for any  $u_0\in H^1_{per}([0,2\pi])$ and $t\in[0,T]$,
\begin{equation}\label{bcbbm9}
\begin{split}
\|\mathcal{A}(u)\|_{H^1_{per}}&  \leq \|u_0\|_{H^1_{per}}+ \int_0^t \|K*(u+|u|^{\frac{3}{2}})\|_{H^1_{per}}d\tau \\
& \leq  \|u_0\|_{H^1_{per}}+ T\left( \|u\|_{X_T}+ C_0\|u\|^{\frac{3}{2}}_{X_T}\right).
\end{split}
\end{equation}

On the other hand,  \eqref{bcbbm7} and the Mean Value Theorem implies the existence of $\theta\in(0,1)$ such that, for all $u, v \in X_T$,
$$
\|K* (|u|^{\frac{3}{2}}-|v|^{\frac{3}{2}})\|^2_{H^1_{per}} \ds \leq \frac{9}{4}\| |\theta v +(1-\theta)u|^\frac{1}{2}(u-v)\|^2_{L^2_{per}} \leq C_1(\|v\|_{X_T}+\|u\|_{X_T})\|u-v\|^2_{H^1_{per}},
$$
that is, there exists $C_1>0$ such that for all $u, v \in X_T$,
\begin{equation}
\|K* (|u|^{\frac{3}{2}}-|v|^{\frac{3}{2}})\|_{H^1_{per}}\leq C_1(\|v\|_{X_T}+\|u\|_{X_T})^\frac{1}{2}\|u-v\|_{H^1_{per}}.
\label{bcbbm10}
\end{equation}
Thus, by using  \eqref{bcbbm10} and arguing as in \eqref{bcbbm9}, we deduce
\begin{equation}
\|\mathcal{A}u-\mathcal{A}v\|_{X_T}\leq  T\|u-v\|_{X_T}\left(1+C_1(\|v\|_{X_T}+\|u\|_{X_T})^\frac{1}{2}\right).
\label{bcbbm11}
\end{equation}
Finally, if $u,  v \in B_{r_0}$, we have from
\eqref{bcbbm9} and \eqref{bcbbm11},
$$\|\mathcal{A}u\|_{X_T}\leq \|u_0\|_{H^1_{per}}+ T\left( r_0+ C_0r_0^{\frac{3}{2}}\right) \quad \mbox{and}\quad \|\mathcal{A}u-\mathcal{A}v\|_{X_T}\leq  T\|u-v\|_{X_T}\left(1+C_1(2r_0)^\frac{1}{2}\right).
$$
By choosing  $r_0:=2\|u_0\|_{H^1_{per}}$ and $\ds T:=\frac{1}{2}(1+\tilde{C}\sqrt{r_0})^{-1}$, where $\tilde{C}=\max\left\{ C_0,\sqrt{2}C_1\right\}$, we get
$$\|\mathcal{A}u\|_{X_T}\leq \frac{r_0}{2}+ Tr_0\left( 1+ C_0r_0^{\frac{1}{2}}\right)\leq \frac{r_0}{2}+ r_0\frac{1}{2}(1+C_0\sqrt{r_0})^{-1}\left( 1+ C_0r_0^{\frac{1}{2}}\right)=\frac{r_0}{2}+\frac{r_0}{2}=r_0
$$
and
$$
\|\mathcal{A}u-\mathcal{A}v\|_{X_T}\leq  T\|u-v\|_{X_T}\left(1+C_1(2r_0)^\frac{1}{2}\right)\leq \frac{1}{2}\frac{\left(1+C_1(2r_0)^\frac{1}{2}\right)}{(1+C_1\sqrt{2r_0})}\|u-v\|_{X_T}=\frac{1}{2}\|u-v\|_{X_T}.
$$
It follows from the last two inequalities that $\mathcal{A}: B_{r_0}\longrightarrow B_{r_0}$ is well-defined and is a contraction. Consequently, there exists a unique $u\in B_{r_0}$ such that  $\mathcal{A}u(t)=u(t)$, for all $t\in [0,T]$. \\

\noindent \textbf{Uniqueness:} Suppose  $u$ and  $v$ are two solutions of  \eqref{bcbbm3}, defined on a common interval, say, $[0,T]$, with initial data $u_0$ and $v_0$, respectively. Arguing as before,   we obtain
$$
\|u(t)-v(t)\|_{H^1_{per}}
\ds \leq\|u_0-v_0\|_{H^1_{per}} + \left(1+C_1(\|v\|_{X_T}+\|u\|_{X_T})^\frac{1}{2}\right)\int_0^t \|u(\tau)-v(\tau)\|_{H^1_{per}}d\tau.
$$
Therefore, from Gronwall's inequality, we see that
\begin{equation}\|u(t)-v(t)\|_{H^1_{per}}\leq \|u_0-v_0\|_{H^1_{per}}\mbox{e}^{\left(1+C_1(\|v\|_{X_T}+\|u\|_{X_T})^\frac{1}{2}\right)t},
\label{bcbbm12}
\end{equation}
for all $t\in[0,T]$. This implies that $u(t)=v(t)$, $t\in[0,T]$, inasmuch as $u_0=v_0$.\\

\noindent \textbf{Continuous Dependence:} Next we prove  that for any $T'\in(0,T)$ there exists a neighborhood $W$ of $u_0$ in $H^1_{per}([0,2\pi])$ such that the map $G$ that associates to each $v_0\in W$ the solution of  \eqref{bcbbm3} with $v_0$ instead of $u_0$ is continuous.

In order  to see that $G$ is well defined, it is needed to show the mapping  $v_0 \mapsto T_{v_0}$ is lower semi-continuous at $u_0$, where  $T_{v_0}:=T(v_0)$ is the existence time of the solution with initial data  $v_0$. In other words, we have to prove there exists a neighborhood  $W$ of $u_0$, so that for every $v_0\in W$, the solution $v$ with initial data $v_0$ exists for all  $t\in[0,T']$. 
 
With this in mind, we define $W$ as
$$W:= \left\{ v_0 \in H^1_{per}([0,2\pi]); \|v_0-u_0\|_{H^1_{per}}<\delta, \quad \mbox{with} \,\ 0<\delta< \frac{1}{8\tilde{C}^2}\left( \frac{1}{T'}-\frac{1}{T}\right)^2 \right\},$$
where $\tilde{C}$ is defined as before. Note that, if $v_0\in W$, then
$$\|v_0\|_{H^1_{per}}\leq \|v_0-u_0\|_{H^1_{per}}+\|u_0\|_{H^1_{per}} \leq \delta + \|u_0\|_{H^1_{per}}.$$
Let $s_0:=2\|v_0\|_{H^1_{per}}$. From the local existence, we have
$$
\begin{array}{ll}
\ds\frac{1}{2T_{v_0}}=1+\tilde{C}\sqrt{s_0}&\leq 1+ \tilde{C}\sqrt{2\delta + 2 \|u_0\|_{H^1_{per}}} \leq 1+ \tilde{C}\sqrt{2\delta}+\tilde{C}\sqrt{2\|u_0\|_{H^1_{per}}}\\
\\
&= \ds 1+\tilde{C}\sqrt{r_0}+\tilde{C}\sqrt{2\delta} < \frac{1}{2T} + \sqrt{2}\tilde{C}\sqrt{\frac{1}{8\tilde{C}^2}\left( \frac{1}{T'}-\frac{1}{T}\right)^2}\\
\\
&\ds \leq \frac{1}{2T} + \sqrt{2}\tilde{C}\frac{1}{2\sqrt{2}\tilde{C}}\left( \frac{1}{T'}-\frac{1}{T}\right)=\frac{1}{2T'}.
\end{array}
$$
Therefore, the existence time $T_{v_0}$ must satisfy $T_{v_0}> T'$
and the function $G$ is well defined. The continuity of $G$ follows from \eqref{bcbbm12}.
The proof of the proposition is thus completed.
\end{proof}

We now can state the main result of this section.

\begin{theorem}[Global well-posedness]\label{gwellp}
For every $u_0\in H^1_{per}([0,2\pi])$,  the solution $u$ obtained in Proposition \ref{lowellp} can be extended for all $t\geq0$. In particular  $u\in C([0,\infty);H^1_{per}([0,\pi]))$.
\end{theorem}
\begin{proof} The proof is based on an iterative process taking into account  the conservation of the quantity $Q$ in \eqref{quantitiesconserved}.  Define $v_0=u(T)$ and consider the integral equation
\begin{equation}
v(t)=v_0+\int_0^t K*(v+|v|^{\frac{3}{2}})(\tau)d\tau.
\label{bcbbm13}
\end{equation}
From Proposition \ref{lowellp}, it follows that there exist  $\tilde{T}>0$  and a unique solution  $v$ for  \eqref{bcbbm13} satisfying $v \in C([0,\tilde{T}];H^1_{per})$. By defining 
$$w(t)= \left\{ \begin{array}{l}
\ds u(t), \,\ \mbox{if}\,\ 0\leq t \leq T,\\
\ds v(t-T), \,\ \mbox{if}\,\ T\leq t \leq T+\tilde{T},
\end{array}\right.
$$
we see that, for all $T< s \leq \tilde{T}$,
$$\begin{array}{rl}
w(T+s)&  = v(s)\\
&=\,\  \ds v_0+\int_0^s K * (v+|v|^{\frac{3}{2}})(\tau)d\tau\\
&=\,\ \ds  u_0 + u(T) - u_0 + \int_0^s K * (v+|v|^{\frac{3}{2}})(\tau)d\tau .\\
\end{array}$$
By making the change of variable $\xi-T=\tau$, we obtain
$$\begin{array}{rl}
w(T+s)&=\ds  u_0 + u(T) - u_0 +\int_T^{T+s} K * (v+|v|^{\frac{3}{2}})(\xi-T)d\xi\\
& \ds = \,\ u_0 + \int_0^T K * (u+|u|^{\frac{3}{2}})(\tau)d\tau + \int_T^{T+s} K * (v+|v|^{\frac{3}{2}})(\xi-T)d\xi\\
&\ds =\,\  u_0 + \int_0^{T+s} K * (w+|w|^{\frac{3}{2}})( \tau)d\tau,
\end{array}
$$
for all $T< s \leq \tilde{T}$. This implies that $u$ can be extended to the interval $[0,T+\tilde{T}]$ and
$$
2\tilde{T}=\frac{1}{1+\tilde{C}\sqrt{2\|v_0\|_{H^1_{per}}}}=\frac{1}{1+\tilde{C}\sqrt{2\|u(T)\|_{H^1_{per}}}}=\frac{1}{1+\tilde{C}\sqrt{2\|u_0\|_{H^1_{per}}}}=2T,
$$
where we used the conserved quantity in \eqref{quantitiesconserved}.
Repeating this process we get $u\in C([0,\infty),H^1_{per}([0,2\pi]))$. This completes the proof of the theorem.
\end{proof}

\section{Existence of an explicit curve of periodic solutions}\label{section3}

 Our goal in this section is to explicit a smooth curve of periodic solution for \eqref{el} with $A\neq0$. The plan to accomplish this consists in using the well known \textit{quadrature method}. Multiplying \eqref{el} by $\phi'$ and integrating once, we get 
\begin{equation}
\frac{(1-c)}{2}\phi^2+\frac{c}{2}(\phi ')^2+\frac{2}{5}\phi^{\frac{5}{2}}-A\phi =B,
\label{eqquad}
\end{equation}
where $B$ is an integration constant. Throughout our construction, we assume $B=0$. Since we are interested in positive solutions, we change  variables by defining $\psi$ through the relation $\phi=\psi^2$. So, $\psi$ must satisfy 
\begin{equation}(\psi ')^2=\frac{1}{5c}\left(-\psi^3+\frac{5(c-1)}{4}\psi^2+\frac{5A}{2}\right)=\frac{1}{5c} P(\psi),
\label{eqpoli}
\end{equation}
where $P$ is the polynomial $P(x)=-x^3+\frac{5(c-1)}{4}x^2+\frac{5A}{2}$.

Equation \eqref{eqpoli} is very similar to that when we look for periodic traveling waves to the KdV equation (see e.g., \cite{angulo1}). Indeed, since we look for a particular solution of \eqref{eqpoli}, we assume that $P$ has three real roots, say, $\beta_1$, $\beta_2$ and $\beta_1$, satisfying $\beta_1<0<\beta_2<\beta_3$, in such a way we can write  $P(\psi)=(\psi-\beta_1)(\psi-\beta_2)(\beta_3-\psi)$. Thus, it must be the case that
\begin{equation}\left\{\begin{array}{l}
\ds \beta_1+\beta_2+\beta_3=\frac{5(c-1)}{4},\\
\ds \beta_1\beta_2+\beta_1\beta_3+\beta_2\beta_3=0,\\
\ds \beta_1\beta_2\beta_3=\frac{5A}{2},\\
\end{array}\right.
\label{sistemraizes}
\end{equation}
and $\psi$ satisfies
\begin{equation}\label{quadpsi}
(\psi')^2=\frac{1}{5c}(\psi-\beta_1)(\psi-\beta_2)(\beta_3-\psi).
\end{equation}

By using standard properties of the elliptic functions and arguing as in \cite{angulo1}, we see that \eqref{quadpsi} possesses a \textit{cnoidal-type solution}, namely,
\begin{equation}\label{cnoidalsol}
\psi(z)=\beta_2+(\beta_3-\beta_2)\cn^2\left(\sqrt{\frac{\beta_3-\beta_1}{16( \beta_1 +\beta_2+\beta_3)+20}} \,\ z,k\right),
\end{equation}
where the elliptic modulus is defined by
\begin{equation}\label{ellipk}
k^2:=\frac{\beta_3-\beta_2}{\beta_3-\beta_1}.
\end{equation}

\begin{remark}\label{remB}
In order to ensure the existence of periodic solutions for \eqref{eqquad}, it suffices to assume that 
$$
H(\phi)=\frac{(1-c)}{2}\phi^2+\frac{2}{5c}\phi^{\frac{5}{2}}-A\phi
$$
has a local minimum (see e.g., \cite{jack}). Thus, for a suitable choice of the parameters $c,A$, and $B$, we can also obtain periodic solutions. The main reason why we choose $B=0$ and find the solution in \eqref{cnoidalsol} lies on the fact that we can establish the spectral  properties of the linearized operator in a simple way. As we will see below, such operator plays a crucial role in the analysis of orbital stability. When one deals with more general solutions, depending on $(c,A,B)$, it is a difficult task to determine the non-positive spectrum of the linearized operator. See also our Remark 5.8.
\end{remark}

The next lemma provides useful informations on the solutions of system \eqref{sistemraizes} in the situation we need. 

\begin{lemma}\label{lemma1}
If the system 
\begin{equation} \label{sys3.34}
\begin{cases}
\beta_1+\beta_2+\beta_3=\frac{5(c-1)}{4},\\
\beta_1\beta_2+\beta_1\beta_3+\beta_2\beta_3=0,
\end{cases}
\end{equation}
possesses a solution with $\beta_1<0<\beta_2<\beta_3$, then
$$
c-1>0 \quad \mbox{and} \quad 0<\beta_2<\frac{5(c-1)}{6}<\beta_3<\frac{5(c-1)}{4}.
$$
Conversely, if $c>1$ is a fixed constant and $\beta_2$ satisfies
$$
0<\beta_2<\frac{5(c-1)}{6},
$$
then, defining
$$
\beta_3:=\frac{1}{2}\left( \frac{5(c-1)}{4}-\beta_2\right)+\frac{1}{2}\sqrt{\Delta_c(\beta_2)},
$$
and
$$
\beta_1:=\frac{5(c-1)}{4}-\beta_2-\beta_3,
$$
where
$$
\Delta_c(x)=\left(x-\frac{5(c-1)}{4}\right)^2-4\left(x^2-\frac{5(c-1)}{4}x\right)^2,
$$
it follows that $(\beta,\beta_2,\beta_3)$ is a solution of \eqref{sys3.34}  satisfying  $\beta_1<0<\beta_2<\beta_3$.
\end{lemma}
\begin{proof}
The proof relies on straightforward algebraic calculations, so we omit the details. 
\end{proof}

\begin{remark}\label{rem3.2}
It is clear if $\beta_1,\beta_2,\beta_3$ is a solution of \eqref{sys3.34} then it also is a solution of \eqref{sistemraizes}, inasmuch as we define $A$ to be $2\beta_1\beta_2\beta_3/5$.
\end{remark}

Note that the solution in \eqref{cnoidalsol} depends on various parameters. However, with the construction in Lemma \ref{lemma1} at hand, we can view it as a function depending only on the parameters $c$ and $\beta_2$. As a result, we obtain the following.

\begin{corollary}\label{corexis}
Let
$$
\mathcal{P}:=\left\{(c,\beta_2)\in \mathbb{R}^2;\,\ c>1\quad \mbox{and} \quad 0<\beta_2<\frac{5(c-1)}{6} \right\}.
$$
For any $(c,\beta_2)\in\mathcal{P}$, let $\beta_1,\beta_3$ be as in Lemma \ref{lemma1}. Then,  the function $\phi_{(c,\beta_2)}=\phi$ defined by
\begin{equation}\label{eq7}
\phi(z)=\left[\beta_2+(\beta_3-\beta_2)\mathrm{CN}^2\left(\sqrt{\frac{\beta_3-\beta_1}{16( \beta_1 +\beta_2+\beta_3)+20}} \,\ z;k\right)\right]^2,
\end{equation} 
with $k$ as in \eqref{ellipk}, solves \eqref{el}.
\end{corollary}
\begin{proof}
By using \eqref{cnoidalsol} and Lemma \ref{lemma1}, it suffices to recall that $\phi=\psi^2$.
\end{proof}

Next, we pay particular attention to the period of the function in \eqref{eq7}. In view of the relations in Lemma \ref{lemma1}, it is not difficult to check that  the period of $\phi$ is
\begin{equation}
T_{\phi}(c,\beta_2):= \frac{4\sqrt{5c}}{(\Delta_c(\beta_2))^\frac{1}{4}}K(k(c,\beta_2)),
\label{periodophi}
\end{equation}
where $K(k)$ is the complete elliptic integral of the first kind and
\begin{equation}\label{kcbeta}
k(c,\beta_2)=\sqrt{\frac{5(c-1)-12\beta_2}{8(\Delta_c(\beta_2))^\frac{1}{2}}+\frac{1}{2}}.
\end{equation}

In order to see that the period in \eqref{periodophi} ranges over an unbounded interval, we establish the following.

\begin{lemma}\label{deriper}
Let $T$ be the function that associates to each $(c,\beta_2)\in\mathcal{P}$ the period of the function $\phi=\phi_{(c,\beta_2)}$, that is,
$$
\begin{array}{lcll}
T:& \mathcal{P} & \longrightarrow & \mathbb{R}\\
& \ds (c,\beta_2)& \ds \longmapsto &\ds T(c,\beta_2)=T_\phi(c,\beta_2),
\end{array}
$$
where $T_\phi(c,\beta_2)$ is defined in \eqref{periodophi}. Then, for all $(c,\beta)\in \mathcal{P}$, $\frac{\partial T}{\partial\beta_2}(c,\beta_2)<0$.
\end{lemma}
\begin{proof}
First notice that \eqref{kcbeta} provides the relations
\begin{equation}
16\Delta_c(\beta_2)(2k^2-1)^2=(5(c-1)-12\beta_2)^2
\label{eq11}
\end{equation}
and
\begin{equation}
2\Delta_c(\beta_2)^{\frac{1}{2}}(2k^2-1)=\frac{5(c-1)}{2}-6\beta_2.
\label{eq12}
\end{equation}
By taking the derivative of  $T$ with respect to $\beta_2$, we obtain
$$\ds\frac{\partial T}{\partial \beta_2}=\gamma_2\left[4\Delta_c(\beta_2)\frac{\partial K}{\partial k} \frac{\partial k}{\partial\beta_2}-K(k)\frac{\partial}{\partial\beta_2}\Delta_c (\beta_2)\right],$$
where $\ds\gamma_2:=\frac{\Delta_c(\beta_2)^{-\frac{3}{4}}\sqrt{5c}}{\Delta_c(\beta_2)^{ \frac{1}{2}}}$. 
By observing that 
$$
\frac{dK}{dk}=\frac{E(k)-(1-k^2)K(k)}{k(1-k^2)},\quad
\frac{\partial k}{\partial \beta_2}=\frac{1}{32\Delta_c(\beta_2)^{\frac{1}{2}}}\left[ -24\Delta_c(\beta_2)-\frac{1}{2}(5(c-1)-12\beta_2)^2 \right],
$$
and using relations (\ref{eq11}) and  (\ref{eq12}), we conclude that
$$\ds\frac{\partial T}{\partial \beta_2}=\gamma_3\left[(-4k^4+4k^2-4)E(k) +(2k^4-6k^2+4)K(k)\right]=:\gamma_3s(k),$$
with $\ds\gamma_3=\frac{\gamma_2\Delta_c(\beta_2)^{\frac{1}{2}}}{k^2(1-k^2)}>0$.

Thus, in order to prove that $\ds\frac{\partial T}{\partial \beta_2}<0$, it is enough to show that $s(k)<0$. But, since   $s(k)\rightarrow 0$, as $k\rightarrow 0$, it suffices to prove that $s'(k)<0$, for any $k\in(0,1)$. Notice that
$$s'(k)=\frac{10k^2}{k}[(1-2k^2)E(k)+(k^2-1)K(k)].$$
So, $s'(k)<0$ if and only if $(1-2k^2)E(k)+(k^2-1)K(k)<0$, which  is equivalent to $$(1-2k^2)E(k)<(1-k^2)K(k).$$
This last inequality follows because $1-2k^2<1-k^2$, $(1-k^2)>0$, $E(k)>0$, $K(k)>0$ and $E(k)<K(k)$. The proof is thus completed.
\end{proof}

Now fix any $c>1$. It follows from Lemma \ref{deriper} that the function $\beta_2\mapsto T_\phi(\beta_2):=T(c,\beta_2)$ is strictly decreasing on the interval $\left(0,\frac{5(c-1)}{6}\right)$. Moreover, \eqref{kcbeta} and \eqref{periodophi} yield 
$$
T_{\phi}(\beta_2)\longrightarrow +\infty, \quad \mbox{as}\;\;\beta_2 \longrightarrow 0,
$$
and 
$$
T_{\phi}(\beta_2)\longrightarrow 4\pi\sqrt{\frac{c}{c-1}}, \quad \mbox{as}\;\;\beta_2\longrightarrow \frac{5(c-1)}{6},
$$
where we have used that $K(0)=\pi/2$ and $K(k)\to+\infty$, as $k\to1$.
 A consequence of the above analysis is the following: fix any real number $L>4\pi$ and choose a constant $c>1$ satisfying $ 4\pi\sqrt{\frac{c}{c-1}}<L$; then there exists a unique $\beta_2\in \left(0,\frac{5(c-1)}{6}\right)$ such that $T_{\phi}(\beta_2)=L$.

Gathering all the informations above together, we can establish the main result of this section, which reads as follows.

\begin{theorem}\label{mainteoex}
Fix any number $L> 4\pi$ and choose  $c_0>\frac{L^2}{L^2-16\pi^2}\ds$. Let  $\beta_2^0:=\beta_2(c_0)\in\left(0,\frac{5(c_0-1)}{6}\right)\ds$ be the unique number  such that $T_{\phi}(\beta_2^0)=L$. Then, the following statements hold.
\begin{enumerate}
\item [(i)] There exist intervals $J(c_0)$ and $B(\beta_2^0)$ around $c_0$ and $\beta_2^0$, respectively, and a unique function $\Lambda:J(c_0)\to B(\beta_2^0)$ such that $\Lambda(c_0)=\beta_2^0$ and 
$$T_{\phi}(c,\beta_2)=\frac{4\sqrt{5c}}{(\Delta_c(\beta_2))^\frac{1}{4}}K(k) =L,$$
where $c\in J(c_0)$, $\beta_2=\Lambda(c)$, and $k^2=k^2(c,\beta_2)=k^2(c,\Lambda(c))  \in(0,1)$ is as in \eqref{kcbeta}.
\item [(ii)] Let  $\beta_2=\Lambda(c)$ and define $\beta_1=\beta_1(c)$ and $\beta_3=\beta_3(c)$ as in Lemma \ref{lemma1}. Then, the cnoidal wave  solution of $(\ref{el})$, defined in $(\ref{eq7})$, determined now by $\beta_1(c)$, $\beta_2(c)$ and $\beta_3(c)$, has fundamental period $L$ for all $c\in J(c_0)$. Moreover, the mapping
$$
c\in J(c_0)\longmapsto \phi_c:=\phi_{(c,\Lambda(c))}\in C^{\infty}_{per}([0,L]),
$$
is a smooth function.
\item [(iii)] $J(c_0)$ can be extended to $\ds\left(\frac{L^2}{L^2-16\pi^2},+\infty\right)$.
\end{enumerate}
\end{theorem}
\begin{proof}
The proof of (i) follows as an application of the implicit function theorem, taking into account Lemma \ref{deriper}. Parts (ii) and (iii) follow immediately from part (i) and the arbitrariness of $c_0$. The interested reader will find some details, for instance in \cite{angulo5}, \cite{angulo3}, or \cite{natali-pastor}.
\end{proof}

In order to make the computations  in the proof of the orbital stability easier, we reparameterize, for each $L>4\pi$ fixed, the smooth curve of periodic solutions obtained in the previous theorem as a function of the modulus $k$. To do so, we note the parameters $c$ and $\beta_2(c)$ can be viewed as functions of $k$.   In fact, let $L>4\pi$ be fixed. After some algebraic manipulations with the relations 
\begin{equation}\left\{
\begin{array}{l}
\ds L=\frac{4\sqrt{5c}}{\Delta_c(\beta_2(c))^{\frac{1}{4}}}K(k),\\
\\
\ds k^2=\frac{5(c-1)-12\beta_2(c)}{8\Delta_c(\beta_2(c))^{\frac{1}{2}}}+\frac{1}{2},\\
\\
\ds \Delta_c(\beta_2(c))=\left(\beta_2(c)-\frac{5(c-1)}{4}\right)^2-4\left(\beta_2^2(c)- \frac{5(c-1)}{4}\beta_2(c)\right),
\end{array}\right.
\label{sistemack}
\end{equation}
given in Theorem \ref{mainteoex}, we obtain
\begin{equation}
(L^2-64 K^2(k)\sqrt{k^4-k^2+1})c=L^2.
\label{ck1}
\end{equation}
Observe that the function $k\in(0,1)\mapsto g(k):=64 K^2(k)\sqrt{k^4-k^2+1}$ is strictly increasing, $g(0)=16\pi^2$ and $g(k)\to+\infty$, as $k\to1$. Thus, there exists a unique $k=k_L$ satisfying $g(k_L)=L^2$. In view of \eqref{ck1},
\begin{equation}
c=\frac{L^2}{L^2-64 K^2(k)\sqrt{k^4-k^2+1}}, \qquad k\in(0,k_L).
\label{ck}
\end{equation}
In a similar fashion, we deduce that
\begin{equation}
\ds\beta_2=\left[\frac{L^2-64(2k^2-1)K^2(k)}{L^2-64K^2(k)\sqrt{ k^4-k^2+1}} -1\right], \quad k\in(0,k_L).
\label{beta2kl}
\end{equation}
Also, recalling that we have defined the constant $A$ to be $2\beta_1\beta_2\beta_3/5$, we have
\begin{equation}
A=\frac{-204800K^6(k)}{27\tilde{m}^3(k)}\left[\left( \sqrt{k^4-k^2+1}-(2k^2-1)\right)^2\left(2\sqrt{k^4-k^2+1}+(2k^2-1)\right)\right],
\label{AkL}
\end{equation}
where $\tilde{m}(k)=L^2-64K^2(k) \sqrt{k^4-k^2+1}$.

The existence of the smooth curve of periodic traveling waves depending on the elliptic modulus $k$ is stated next for the sake of future references.

\begin{corollary}\label{lascor}
Fix $L>4\pi$ and let $k_L$ be defined as above. For each $k\in(0,k_L)$, the function
\begin{equation}
\phi_{k}(x)=\left[\!\frac{5}{12}\!\left(\!\frac{L^2\!-\!64(2k^2\!-\!1)K^2(k)}{\tilde{m}(k)}-1\! \right)\!+\!\frac{80k^2K^2(k)}{\tilde{m}(k)}\mathrm{CN}^2 \left(\!\frac{2K(k)}{L} x,k\!\right)\! \right]^2,
\label{phikL}
\end{equation}
where $\tilde{m}(k)=L^2-64K^2(k) \sqrt{k^4-k^2+1}$, is a smooth $L$-periodic solution of \eqref{el}. In addition, the mapping
$$
k\in(0,k_L)\mapsto \phi_{k}\in C^\infty_{per}([0,L])
$$ 
is smooth.
\end{corollary}
\begin{proof}
Combine the definition of $\beta_2$ in \eqref{phikL} with Theorem \ref{mainteoex}, using the explicit form of $\phi=\phi_{(c,\beta_2)}$ in \eqref{eq7}.
\end{proof}

It should be noted that Corollary \ref{lascor} establishes property $(P_0)$ proposed in the introduction with $J=(0,k_L)$.

\section{Spectral Analysis}  \label{section4}

Let $\phi=\phi_k$ be an $L$-periodic solution of \eqref{el} given in Corollary \ref{lascor}. Let $\mathcal{L}_k$ be the operator defined on $H^2_{per}([0,L])$ by
\begin{equation*}
\mathcal{L}_kf=-cf''+(c-1)f-\frac{3}{2}\phi_k^{\frac{1}{2}}f.
\end{equation*}
In this section, we show that properties $(P_1)$ and $(P_2)$ hold. It is well known from Floquet's theory (see e.g., \cite{ea}) that the spectrum of $\mathcal{L}_k$ is composed by a sequence of real numbers, say, $\{\lambda_n\}$, $n=0,1,\ldots$, satisfying $\lambda_n\to\infty$, as $n\to\infty$, and
\begin{equation}\lambda_0 < \lambda_1 \leq \lambda_2 \leq \lambda_3 \leq \lambda_4 \leq \ldots \leq \lambda_{2n-1}\leq \lambda_{2n} \ldots,
\label{autovaloresbbm}
\end{equation}
where equality means the eigenvalue is double. Furthermore, if $\chi_n$ denotes an eigenfunction associated with the eigenvalue $\lambda_n$, then $\chi_0$ has no zeros in $[0,L]$, and $\chi_{2n+1}$ and $\chi_{2n+2}$ has exactly $2n+2$ zeros in $[0,L)$.

Taking the derivative with respect to $x$  in \eqref{el} (with $\phi$ replaced by  $\phi_k$), we see that zero is an eigenvalue of  $\mathcal{L}_k$ with associated eigenfunction $\phi_k'$. Also, a careful looking in the explicit form of $\phi_k$ reveals that $\phi_k'$ has exactly two zeros in the interval $[0,L)$. Thus, it must be the case that zero is the second or third eigenvalue in the list \eqref{autovaloresbbm}. In order to show that $(P_1)$ holds, we need to establish that it is the second one.

We start by considering the periodic eigenvalue problem associated with  $\tilde{\mathcal{L}}_k=(1/c)\mathcal{L}_k$, that is,
\begin{equation}\left\{\begin{array}{l}
\ds-\chi''(x)+\left[\frac{(c-1)}{c}-\frac{3}{2c}\phi_k^{\frac{1}{2}}\right]\chi(x)=\alpha \chi(x),\\
\chi(0)=\chi(L), \quad \chi'(0)=\chi'(L).
\end{array}\right.
\label{lame1}
\end{equation}
It is clear that $\alpha$ is an eigenvalue of $\tilde{\mathcal{L}}_k$ if and only if $\lambda=c\alpha$ is an eigenvalue of $\mathcal{L}_k$.
For short, we write $\phi_k$ as 
\begin{equation}\label{phisimpli}
\phi_k(x)=a^2\left[b+ \cn^2\left(\frac{2K(k)}{L}x,k\right)\right]^2,
\end{equation}
where
\begin{equation}
a=\frac{80k^2K^2(k)}{L^2-64K^2(k)\sqrt{k^4-k^2+1}}
\label{lame2}
\end{equation}
and
\begin{equation}
b=\frac{\sqrt{k^4-k^2+1}-2k^2+1}{3k^2}.
\label{lame3}
\end{equation}
With this notation, the eigenvalue problem \eqref{lame1} becomes
\begin{equation}
\left\{\begin{array}{l}
\ds-\chi''(x)+\left[\frac{(c-1)}{c}-\alpha-\frac{3}{2c}\left( a\left\{ b+\cn^2\left(\frac{2K(k)}{L}x,k\right)\right\}\right)\right]\chi(x)=0,\\
\chi(0)=\chi(L), \quad \chi'(0)=\chi'(L).
\end{array}\right.
\label{lame3.0}
\end{equation}
By making the change of variable $\Lambda(x)=\chi(\eta x)$, with $\eta:=\frac{L}{2K(k)}$,  and substituting in \eqref{lame3.0}, it turns out that 
\begin{equation}\left\{\begin{array}{l}
\ds-\Lambda''(x)+\left[\frac{\eta^2(c-1)}{c}-\eta^2\alpha-\frac{3ab\eta^2}{2c}-\frac{3a\eta^2}{2c}\cn^2\left(x,k\right)\right]\Lambda(x)=0,\\
\Lambda(0)=\Lambda(2K(k)), \quad \Lambda'(0)=\Lambda'( 2K(k)).
\end{array}\right.
\label{lame3.1}
\end{equation}
But, using the definitions of $c$, $a$ and $b$ in \eqref{ck}, \eqref{lame2} and \eqref{lame3}, respectively, we have
\begin{equation}\label{lame4}
\begin{split}
\frac{3a\eta^2}{2c}&=\frac{3}{2}\left(\frac{80k^2K^2(k)}{L^2-64K^2\sqrt{k^4-k^2+1}}\right)\left(\frac{L^2}{4K^2(k)}\right)\left(\frac{L^2-64K^2(k)\sqrt{k^4-k^2+1}}{L^2}\right)\\
&=5\cdot 6\cdot k^2,
\end{split}
\end{equation}

\begin{equation}\label{lame5}
\begin{split}
\frac{\eta^2(c-1)}{c}& =\frac{L^2}{4K^2(k)}\left(\frac{64K^2\sqrt{k^4-k^2+1}}{L^2-64K^2(k)\sqrt{k^4-k^2+1}}\right)\left(\frac{L^2-64K^2(k)\sqrt{k^4-k^2+1}}{L^2}\right)
\\
& =16\sqrt{k^4-k^2+1},
\end{split}
\end{equation}
and
\begin{equation}
\frac{3ab\eta^2}{2c}=30k^2b=10(\sqrt{k^4-k^2+1}-2k^2+1).
\label{lame6}
\end{equation}
Thus, replacing \eqref{lame4}-\eqref{lame6} in \eqref{lame3.1} and using the relation $\sn^2+\cn^2=1$, we obtain
\begin{equation}\left\{\begin{array}{l}
\ds\Lambda''(x)+\left[h-5\cdot 6\cdot k^2 \sn^2\left(x,k\right)\right]\Lambda(x)=0,\\
\Lambda(0)=\Lambda(2K(k)), \quad \Lambda'(0)=\Lambda'( 2K(k)),
\end{array}\right.
\label{lame7}
\end{equation}
where
\begin{equation}
h=10(k^2+1)+\eta^2\alpha-6\sqrt{k^4-k^2+1}.
\label{lame8}
\end{equation}
The differential equation in \eqref{lame7} is known as the Lam\'e equation, which arises from the theory of potential of an ellipsoid. Its eigenfunctions are given in terms of the so-called Lam\'e polynomials.

\subsection{Spectral Analysis of the problem \eqref{lame7}}

In this subsection, we study the eigenvalue problem \eqref{lame7}, by showing  we can use the technique developed by Ince \cite{ince} to give the first five eigenvalues and its respective eigenfunctions. In fact, according to \cite[page 53]{ince}, the function
$$
\Lambda(x)=\cn(x)\dn(x)\left(\sn(x)+D_3\sn^3(x)\right),
$$
with
$$ D_3=\frac{1}{6}(4k^2+4-h)$$
is an eigenfunction of \eqref{lame7} associated with the eigenvalue $h$, provided $h$ is a root of the quadratic equation 
\begin{equation}\label{qualame}
h^2-20(1+k^2)h+64(1+k^2)^2+108k^2=0.
\end{equation}
Here and hereafter in this section, for the sake of shortness, we drop the parameter $k$ in the elliptic functions.
Since
\begin{equation}\label{h1}
h_1=10(1+k^2)+6\sqrt{k^4-k^2+1}
\end{equation}
and
\begin{equation}\label{h2}
h_2=10(1+k^2)-6\sqrt{k^4-k^2+1},
\end{equation}\label{Lambda1}
are two real roots of \eqref{qualame}, it follows that 
\begin{equation}
\Lambda_1(x)=\cn(x)\dn(x)\left(\sn(x)-\left(\sqrt{k^4-k^2+1}+(1+k^2)\right)\sn^3(x)\right)
\end{equation}
and
\begin{equation}\label{Lambda2}
\Lambda_2(x)=\cn(x)\dn(x)\left(\sn(x)+\left(\sqrt{k^4-k^2+1}-(1+k^2)\right)\sn^3(x)\right),
\end{equation}
are two eigenfunctions to \eqref{lame7}. Also according to \cite[page 49]{ince}, the function
$$
\Lambda(x)=\dn(x,k)\left(1+C_2\sn^2(x,k)+ C_4\sn^4(x,k)\right),
$$
with
$$ C_2=\frac{1}{2}(h-k^2) \quad\mbox{and}\quad C_4=\frac{1}{12}\left(28k^2-C_2(h-4-9k^2)\right)$$
is an eigenfunction of \eqref{lame7} associated with the eigenvalue $h$, provided now  $h$ is a root of the cubic equation
\begin{equation}
h^3-(20+35k^2)h^2+(64+576k^2+259k^4)h-225k^6-1860k^4-960k^2=0.
\label{lame9}
\end{equation}
Equation \eqref{lame9} can be solved by using, for instance, the \textit{trigonometric method}. Indeed,  by setting
$$
\left\{\begin{array}{l}
z_1=-(20+35k^2),\\
z_2=64+576k^2+259k^4,\\
z_3=-225k^6-1860k^4-960k^2.
\end{array}\right.
$$
one can see that
\begin{equation}\label{h3}
h_3=2\sqrt{\frac{3z_2-z_1^2}{-9}}\mbox{cos}\left(\frac{1}{3}\mbox{arccos}\left(\frac{2z_1^3-9z_1z_2+27z_3}{6(3z_2-z_1^2)}\sqrt{\frac{-9}{3z_2-z_1^2}}\right)-\frac{4\pi}{3}\right)-\frac{z_1}{3},
\end{equation}
\begin{equation}\label{h4}
h_4=2\sqrt{\frac{3z_2-z_1^2}{-9}}\mbox{cos}\left(\frac{1}{3}\mbox{arccos}\left(\frac{2z_1^3-9z_1z_2+27z_3}{6(3z_2-z_1^2)}\sqrt{\frac{-9}{3z_2-z_1^2}}\right)-\frac{2\pi}{3}\right)-\frac{z_1}{3},
\end{equation}
 and
\begin{equation}\label{h5}
 h_5=2\sqrt{\frac{3z_2-z_1^2}{-9}}\mbox{cos}\left(\frac{1}{3}\mbox{arccos}\left(\frac{2z_1^3-9z_1z_2+27z_3}{6(3z_2-z_1^2)}\sqrt{\frac{-9}{3z_2-z_1^2}}\right)\right)-\frac{z_1}{3}.
\end{equation}
 solve \eqref{lame9}. Consequently, the functions
\begin{equation}\label{Lambda3}
\begin{split}
\Lambda_3&(x)=\dn(x)\\
&\times\left(1+\frac{1}{2}(h_3-k^2)\sn^2(x)+ \frac{1}{12}\left(28k^2-\frac{1}{2}(h_3-k^2)(h_3-4-9k^2)\right)\sn^4(x)\right),
\end{split}
\end{equation}
 \begin{equation}\label{Lambda4}
 \begin{split}
 \Lambda_4&(x)=\dn(x)\\
 &\times\left(1+\frac{1}{2}(h_4-k^2)\sn^2(x)+ \frac{1}{12}\left(28k^2-\frac{1}{2}(h_4-k^2)(h_4-4-9k^2)\right)\sn^4(x)\right),
 \end{split}
 \end{equation}
and
\begin{equation}\label{Lambda5}
\begin{split}
\Lambda_5&(x)=\dn(x)\\
&\times\left(1+\frac{1}{2}(h_5-k^2)\sn^2(x)+ \frac{1}{12}\left(28k^2-\frac{1}{2}(h_5-k^2)(h_5-4-9k^2)\right)\sn^4(x)\right).
\end{split}
\end{equation}
also are eigenfunctions of \eqref{lame7}.

Now, it is not difficult to check that the eigenvalues $h_i,i=1,\ldots,5$, can be ordered in the form
\begin{equation}\label{orderh}
h_3<h_2<h_4<h_1<h_5.
\end{equation}
In addition, in the interval $[0,2K)$,  $\Lambda_3$ has no zeros, $\Lambda_2$ and $\Lambda_4$ have exactly two zeros and $\Lambda_1$ and $\Lambda_5$ have exactly four zeros.

\subsection{Spectral Analysis of the operator $\mathcal{L}_k$} Here we prove that indeed the operator $\mathcal{L}_k$ satisfies properties $(P_1)$ and $(P_2)$. From \eqref{lame8}, we have
$$\alpha=\frac{1}{\eta^2}\left(h+6\sqrt{k^4-k^2+1}-10(k^2+1)\right).$$
Hence, the eigenvalues of  $\mathcal{L}_k$ are of the form
\begin{equation}
\begin{array}{ll}
\lambda&=\ds c\frac{1}{\eta^2}\left(h+6\sqrt{k^4-k^2+1}-10(k^2+1)\right)\\
\\
&\ds =\frac{4K^2(k)}{L^2-64K^2\sqrt{k^4-k^2+1}}\left(h+6\sqrt{k^4-k^2+1}-10(k^2+1)\right).
\end{array}
\label{lame11}
\end{equation}

By replacing in \eqref{lame11}, the $h_i,i=1,\ldots,5$, given  in last subsection, according to the order \eqref{orderh}, we find out that 
$$
\displaystyle \lambda_0=\frac{4K^2(k)}{L^2-64K^2\sqrt{k^4-k^2+1}}\left(h_3+6\sqrt{k^4-k^2+1}-10(k^2+1)\right),
$$
$$
\displaystyle \lambda_1=\frac{4K^2(k)}{L^2-64K^2\sqrt{k^4-k^2+1}}\left(h_2+6\sqrt{k^4-k^2+1}-10(k^2+1)\right)=0,
$$
$$
\displaystyle\lambda_2=\frac{4K^2(k)}{L^2-64K^2\sqrt{k^4-k^2+1}}\left(h_4+6\sqrt{k^4-k^2+1}-10(k^2+1)\right),
$$
$$
\displaystyle\lambda_3=\frac{4K^2(k)}{L^2-64K^2\sqrt{k^4-k^2+1}}\left(h_1+6\sqrt{k^4-k^2+1}-10(k^2+1)\right)
$$
$$
\displaystyle\lambda_4=\frac{4K^2(k)}{L^2-64K^2\sqrt{k^4-k^2+1}}\left(h_5+6\sqrt{k^4-k^2+1}-10(k^2+1)\right),
$$
are the first five eigenvalues of $\mathcal{L}_k$ with associated eigenfunctions given, respectively, by 
\[
\begin{split}
\ds\chi_0(x)=\dn(y)\left(1+\frac{1}{2}(h_3-k^2)\sn^2(y)+ \frac{1}{12}\left(28k^2-\frac{1}{2}(h_3-k^2)(h_3-4-9k^2)\right)\sn^4(y)\right)
\end{split}
\]
\[
\ds\chi_1(x)=\cn(y)\dn(y)\left(\sn(y)+\left(\sqrt{k^4-k^2+1}-(1+k^2)\right)\sn^3(y)\right),
\]
\[
\ds\chi_2(x)=\dn(y)\left(1+\frac{1}{2}(h_4-k^2)\sn^2(y)+ \frac{1}{12}\left(28k^2-\frac{1}{2}(h_4-k^2)(h_4-4-9k^2)\right)\sn^4(y)\right),
\]
\[
\ds\chi_3(x)=\cn(y)\dn(y)\left(\sn(y)-\left(\sqrt{k^4-k^2+1}+(1+k^2)\right)\sn^3(y)\right),
\]
\[
\ds\chi_4(x)=\dn(y)\left(1+\frac{1}{2}(h_5-k^2)\sn^2(y)+ \frac{1}{12}\left(28k^2-\frac{1}{2}(h_5-k^2)(h_5-4-9k^2)\right)\sn^4(y)\right),
\]
where $y=x/\eta$. This means that $\mathcal{L}_k$ has a unique negative eigenvalue which is simple and $\lambda_1=0$ is the second eigenvalue, which also is simple. Properties $(P_1)$ and $(P_2)$ are thus established.

\begin{remark}
According to the constructions above, 
\[
\begin{array}{ll}
\chi_1(x)\! & \!=\!\cn\!\left(\!\frac{2K(k)}{L}x\right)\!\dn\!\left(\!\frac{2K(k)}{L}x\right)\!\!\left[\!\sn\left(\!\frac{2K(k)}{L}x\right)\!+\!\left(\!\sqrt{k^4-k^2+1}\!-\!(1+k^2)\!\right)\sn^3\!\left(\!\frac{2K(k)}{L}x\right)\!\right]\\
\\
&=\ds\frac{-L}{8a^2(b+1)K(k)} \phi_k'=:C_0\phi_k',
\end{array}
\]
which is in agreement with our initial remark that $\phi_k'$ is an eigenfunction of $\mathcal{L}_k$ associated with the eigenvalue zero.
\end{remark}

\begin{remark}
The spectral analysis of the operator $\mathcal{L}_k$ can also be deduced in view of the theory developed recently in \cite{natali2}, \cite{Neves1}, \cite{Neves2}. Indeed, one can check that $\mathcal{L}_k$ is an isoinertial family of operators with inertial index given by the pair $(1,1)$, which means that $\mathcal{L}_k$ has a unique negative eigenvalue and zero is a simple eigenvalue. Since such approach uses numerical arguments, we prefer to show properties $(P_1)$ and $(P_2)$ by using the arguments in view of the Lam\'e equation.
\end{remark}

\section{Orbital stability}\label{section5}

Our   goal in this section  is to prove the orbital stability for the $L$-periodic traveling-wave solutions  we have shown to exist in Corollary \ref{lascor}.   Before proceeding let us make clear what we mean by orbital stability. First of all note that \eqref{bbm} is invariant by spatial translations. Roughly speaking, we say that a solution $\phi$ of \eqref{el}  is orbitally stable if, for each $u_0\in H^1_{per}([0,L])$ close to $\phi$, the global solution given in Theorem \ref{gwellp} with initial data $u_0$ remains close to $\phi$ up to a translation. More precisely.

\begin{definition}
Let $\phi$ be an $L$-periodic solution of \eqref{el}. We say that $\phi$ is orbitally stable by the flow of \eqref{bbm} in  $H^1_{per}([0,L])$ if, for any $\varepsilon>0$, there exists $\delta>0$ such that if $u_0\in H^1_{per}([0,L])$ satisfies
$$
\|u_0-\phi\|_{H^1_{per}}<\delta,
$$
then the solution $u(t)$ of \eqref{bbm}, with initial data $u_0$, exists globally and satisfies
$$
\sup_{t\geq0}\inf_{r\in\R}\|u(t)-\phi(\cdot+r)\|_{H^1_{per}}<\varepsilon.
$$
Otherwise, we say that $\phi$ is $H^1_{per}$-unstable.
\end{definition}

In what follows in this section, $\phi_k$ will be an $L$-periodic solution given in Corollary \ref{lascor}. Also, $c=c(k)$ and $A=A(k)$ will be the functions defined in \eqref{ck} and \eqref{AkL}. Note that $k\in(0,k_L)\mapsto c(k)$ and $k\in(0,k_L)\mapsto A(k)$ are smooth  functions and
\begin{equation}\label{deric}
\frac{\partial c}{\partial k}(k)>0.
\end{equation}
Hence, we can define the smooth functional
$$
M_k(u):=\ds\frac{\partial c}{\partial k}Q(u)+\frac{ \partial A}{\partial k} V(u),
$$
where the derivatives of $c$ and $A$ are evaluated at $k$, and $Q$ and $V$ are given in \eqref{quantitiesconserved} and \eqref{quantitiesconserved1}, respectively.

Our main theorem concerning orbital stability reads as follows.

\begin{theorem}
For each $k\in(0,k_L)$, the periodic traveling wave $\phi_k$
 is orbitally stable by the flow of \eqref{bbm} in $H^1_{per}([0,L])$.
\label{bbmmeantheorem}
\end{theorem}

Before proving Theorem \ref{bbmmeantheorem}, we need a series of lemmas. Let us start by showing that property $(P_3)$ holds.

\begin{lemma} \label{lemma5.3}
Let $\Phi$ be defined as
 $$\Phi:=\left\langle\mathcal{ L}_k\left(\frac{\partial \phi_k}{\partial k}\right),\frac{\partial \phi_k}{\partial k}\right\rangle.
$$
 Then,
\begin{equation} \ds \Phi=-\left( M'_k(\phi_k),\frac{\partial \phi_{k}}{\partial k}\right)_{L^2_{per}} <0.
\label{condicarestabilidade1}
\end{equation}
\end{lemma}
\begin{proof}  First we note that  $Q$ and $V$ have  Fr\'echet derivative at $\phi_k$ given by
\begin{equation}Q'(\phi_k)=\phi_k-{\phi_k}'' \qquad \mbox{and}\qquad V'(\phi_k)=1.\label{eq13.1}\end{equation}
On the other hand, by taking the derivative with respect to $k$ in (\ref{el}), we obtain
$$-\frac{\partial c}{\partial k} \left(\phi_k\right)''-c\left(\frac{\partial \phi_k}{\partial k}\right) ''+\frac{\partial c}{\partial k}\phi_k+(c-1)\frac{\partial \phi_k}{\partial k}-\frac{3}{2}\phi_k^{\frac{1}{2}}\frac{\partial \phi_k}{\partial k}+\frac{\partial A}{\partial k}=0,$$
that is,
\begin{equation}\mathcal{L}_k\left(\frac{\partial\phi_k}{\partial k}\right)+\frac{\partial c}{\partial k}(-\phi_k''+\phi_k)+\frac{\partial A}{\partial k}=0.
\label{eq14}
\end{equation}
Thus, from \eqref{eq13.1} and \eqref{eq14}, we get
$$
\mathcal{L}_k\left(\frac{\partial\phi_k}{\partial k}\right)=-\left(\frac{\partial c}{\partial k}  Q'(\phi_k)+\frac{ \partial A}{\partial k} V'(\phi_k)\right)=-M'_k(\phi_k).
$$
This implies that $\Phi$ can be written as
$$
-\Phi=\left( M'_k(\phi_k),\frac{\partial \phi_{k}}{\partial k}\right)_{L^2_{per}}=\ds \left( \frac{\partial c}{\partial k} \frac{\partial }{\partial k}\left(\frac{1}{2}\int_0^L \left(\phi_k '\right)^2+\phi_k^2\right)dx + \frac{\partial A}{\partial k}\frac{\partial}{\partial k}\int_0^L\phi_k dx\right),
$$
i.e.,
\begin{equation}
-\Phi=\left(\frac{\partial c}{\partial k} \frac{\partial }{\partial k}Q(\phi_k) + \frac{\partial A}{\partial k}\frac{\partial}{\partial k}V(\phi_k)\right).\label{14.1}
\end{equation}

In order to determine the sign of $\Phi$, we take the derivative with respect to $x$ in \eqref{phisimpli} to obtain
\begin{equation}\left(\frac{\partial\phi_k}{\partial x}\right)^2=\frac{64a^4K^2}{L^2}\left( b+\cn^2\right)^2\left(\cn^2\sn^2 \dn^2\right),
\label{phixkL}
\end{equation}
where,  for short, we used the notation $\cn^2:=\cn^2\left(\frac{2K(k)}{L}x,k\right)$, $\sn^2:=\sn^2\left(\frac{2K(k)}{L}x,k\right)$, $\dn^2:=\dn^2\left(\frac{2K(k)}{L}x,k\right)$ and $K^2:=K^2(k)$. Now,  by using the relations $\sn^2+\cn^2=1$ and $\dn^2+k^2\sn^2=1$,
we can express (\ref{phixkL}) as
\[
\begin{split}
\ds\left(\frac{\partial\phi_k}{\partial x}\right)^2 &=\ds \frac{64a^4K^2}{L^2}\Big{(} \big{[}b^2(1-k^2)\big{]}\cn^2+\big{[}2b(1-k^2)  +b^2(2k^2-1)\big{]}\cn^4\\
&\ds\;\;\;+\big{[}(1-k^2)-b^2k^2+2b(2k^2-1)\big{]}\cn^6\ds+ \big{[} (2k^2-1)-2bk^2 \big{]} \cn^8 - k^2\cn^{10}\Big{)}.
\end{split}
\]
Integrating $\phi_k$, $\phi^2_k$ and $\ds\left( \frac{\partial \phi_k}{\partial x}\right)^2$ from 0 to $L$, we have
\begin{equation}\label{14.4}
\ds\int_0^L\phi_k(x)dx=a^2\left[b^2L+2bC_2+C_4\right],
\end{equation}
\begin{equation}\label{14.41}
\ds\int_0^L\phi_k^2(x)dx=a^4\left[b^4L+4b^3C_2+6b^2C_4+4bC_6+C_8\right],
\end{equation}
and
\begin{equation}
\begin{array}{rl}
\ds\int_0^L\!\!\left(\frac{\partial\phi_k}{\partial x}\right)^2\!\!\!(x )dx=&\ds \frac{64a^4K^2}{L^2}\Big{(} \big{[}b^2(1-k^2)\big{]}C_2+\big{[}2b(1-k^2)  +b^2(2k^2-1)\big{]}C_4\\
\\
&\ds+\big{[}(1-k^2)-b^2k^2+2b(2k^2-1)\big{]}C_6\\
\\
&\ds+ \big{[} (2k^2-1)-2bk^2 \big{]} C_8 - k^2C_{10}\Big{)},
\end{array}
\label{14.5}
\end{equation}
where
$$C_{2n}=\int_0^L\cn^{2n}\left(\frac{2K(k)}{L}x,k\right)dx.$$
By making the change of variables $\xi=\frac{2K(k)}{L}x$ and using  that $\cn^{2n}$ é is an even function, we obtain
\begin{equation}
C_{2n}=\frac{L}{K(k)}\int_0^{K(k)}\cn^{2n}(\xi,k)d\xi=:\frac{L}{K(k)}\tilde{ C}_{2n}, \label{14.6}
\end{equation}
where (see  formulas \textsection  $312.02$, \textsection $312.04$ and  \textsection $312.05$ in \cite{friedman})
$$
\tilde{C}_2=\frac{E(k)-(1-k^2)K(k)}{k^2},
$$
$$\tilde{C}_4=\frac{(2-3k^2)(1-k^2)K(k)+2(2k^2-1)E(k)}{3k^4},
$$
and
$$\tilde{C}_{2n+2}=\frac{2n(2k^2-1)\tilde{C}_{2n}+(2n-1)(1-k^2)\tilde{C}_{2n-2}}{ (2n+1)k^2} \quad n=2,3,4.
$$
The quantities $\tilde{C}_2$, $\tilde{C}_4$ and $\tilde{C}_{2n+2}$, for $ n=2,3,4,$  provide an explicit formula for $Q(\phi_k)$ and $V(\phi_k)$ and then for $\Phi$  as well. Indeed, after some calculations involving \eqref{ck}, \eqref{AkL}, \eqref{lame2}, \eqref{lame3}, \eqref{14.4}, \eqref{14.41}, \eqref{14.5}, and \eqref{14.6} one can write 
$$V(\phi_k)=c^2(k)g_1(k),$$
$$Q(\phi_k)=c^4(k)g_2(k),$$
and
$$A=c^3(k)g_3(k),$$
where
$$\begin{array}{l}
\ds g_1(k)=\frac{12800K^3(k)}{9L^3}\left[\left((k^4-k^2+1)+(k^2-2)\sqrt{k^4-k^2+1}\right)K(k)+3\sqrt{k^4-k^2+1}E(k)\right],\\
\\
\ds g_3(k)=-\frac{204800K^6(k)}{27L^6}\left[2k^6-3k^4-3k^2+2+(2k^4-2k^2+2)\sqrt{k^4-k^2+1}\right]\end{array}$$
and $2g_2(k)=h_1(k)+h_2(k)$ with
$$\begin{array}{ll}
h_1(k)&=\, \ds\frac{32768000K^7(k)}{567L^7}\left[ (100-200k^2+225k^4-125k^6+70k^8)K(k)\right.\\ \\
& \hspace{3,4cm} + (70k^6-252k^4+336k^2-224)\sqrt{k^4-k^2+1}K(k)\\ \\
& \hspace{3,4cm} + (-30k^6+45k^4+45k^2-30)E(k)\\ \\
& \hspace{3,4cm} + \left. (294k^4-294k^2+294)\sqrt{k^4-k^2+1}E(k)\right]
\end{array}
$$
and
$$\begin{array}{ll}
h_2(k)&=\, \ds- \frac{1048576000K^9(k)}{189L^9}\left[ (7k^8-28k^6+42k^4-35k^2+14)K(k)\right.\\ \\
& \hspace{4cm} + (-5k^6-5k^4+20k^2-10)\sqrt{k^4-k^2+1}K(k)\\ \\
& \hspace{4cm} + (-14k^8+28k^6-42k^4+28k^2-14)E(k)\\ \\
& \hspace{4cm} + \left. (10k^6-15k^4-15k^2+10)\sqrt{k^4-k^2+1}E(k)\right].
\end{array}
$$
The functions $h_1$ and $h_2$ are obtained in view of \eqref{14.41} and \eqref{14.5},  respectively. Hence, taking derivatives with respect to $k$ in the above functions, we obtain, from \eqref{14.1}, 
$$ 
-\Phi=\ds (4g_2+6g_1g_3)c^3\left(\frac{\partial c}{\partial k}\right)^2 
 +(g_2'+3g_3g_1'+2g_1g_3')c^4\frac{\partial c}{\partial k}+g_1'g_3'c^5. $$
Since $c>1$, we see that $-\Phi>0$ if and only if
\begin{equation}
(4g_2+6g_1g_3)\left(\frac{\partial c}{\partial k}\right)^2+(g_2'+3g_3g_1'+2g_1g_3')c\frac{\partial c}{\partial k}+ g_1'g_3'c^2>0
\label{t1}\end{equation}
Recalling that $\ds \frac{\partial c}{\partial k}>0$, for \eqref{t1} it suffices that
$$f_1:=4g_2+6g_1g_3>0, \quad f_2:=g_2'+3g_3g_1'+2g_1g_3'>0,\quad \mbox{and} \quad f_3:=g_1'g_3'>0.$$

The most difficult calculations are those related to the quantity $f_2$. So, we show next that $f_2(k)>0$. Let us start with the derivatives. Since
$$\frac{d K(k)}{d k}=\frac{E(k)-(1-k^2)K(K)}{k(1-k^2)}\quad \mbox{and}\quad \frac{d E(k)}{d k}=\frac{E(k)-K(k)}{k},$$
we obtain
$$\begin{array}{l}
\ds g_1'(k) = \frac{12800K^2(k)}{9L^3k(k^2-1)\sqrt{k^4-k^2+1}}\left(p_1(k)K^2(k)+p_2(k)E(k)K(k)+p_3(k)E^2(k)\right),\\
\\
\ds g_3'(k)=\frac{409600K^5(k)}{9L^6k(k^2-1)\sqrt{k^4-k^2+1}}\left(p_4(k)K(k)+p_5(k)E(k)\right),\\
\\
\ds h_1'(k)=\frac{32768000K^6(k)}{81L^7k(k^2-1)\sqrt{k^4-k^2+1}} \left(p_6(k)K^2(k)+p_7(k)E(k)K(k)+p_8(k)E^2(k)\right),\\
\end{array}$$
and
$$\begin{array}{l}
\ds h_2'(k) = -\frac{1048576000K^8(k)}{63L^9k(k^2-1)\sqrt{k^4-k^2+1}}\left(p_9(k)K^2(k)+p_{10}(k)E(k)K(k)+p_{11}(k)E^2(k)\right),
\end{array}$$
where $p_i$, $i=0,1,\cdots, 11,$ are well-known functions defined in the Appendix.

As a consequence, we see that $f_2(k)>0$ is equivalent to
$$ \begin{array}{l}
\ds\frac{2^{14}10^3 m_1(k)K^6(k)}{81L^7k(k^2-1)\sqrt{k^4-k^2+1}} +\frac{2^{19}10^3m_2(k)K^8}{63L^9k(k^2-1)\sqrt{k^4-k^2+1}}\\ 
\\
\ds +\frac{2^{18}10^4m_3(k)K^8}{81L^9k(k^2-1)\sqrt{k^4-k^2+1}}+\frac{2^{20}10^4m_4(k)K^8(k)}{81L^9k(k^2-1)\sqrt{k^4-k^2+1}}>0
\end{array}$$
which in turn is equivalent to
\begin{equation}
\frac{L^2m_1(k)}{9}+\frac{32m_2(k)K^2(k)}{7}+\frac{160 m_3(k)K^2(k)}{9}+\frac{640 m_4(k)K^2(k)}{9}<0.
\label{t2}
\end{equation}
The functions $m_i$, $i=1,2,3,4$, depend on $p_i$ and are also defined in the Appendix.

Since $m_1(k)<0$ and $L>4\pi$, in  order to show \eqref{t2}, it suffices that
\begin{equation}
r(k):=\frac{16\pi^2m_1(k)}{9}+\frac{32m_2(k)K^2(k)}{7}+\frac{160 m_3(k)K^2(k)}{9}+\frac{640 m_4(k)K^2(k)}{9}<0.
\label{t3}
\end{equation}

Note that the function $r(k)$ does not depend of the period $L$ and it is a combination of  polynomials with  $K(k)$, $E(k)$, and the function $\sqrt{k^4-k^2+1}$. Thus, we can indeed prove that $r(k)$ is negative for all $k\in(0,1)$.

 Since $r(k)$ is very small for $k$ small, to see that it is negative for $k$ small, we can use the asymptotic expansion of $K(k)$, $E(k),$ and $\sqrt{k^4-k^2+1}$. In fact, recalling  that (see e.g., \cite{friedman})
$$
K(k)=\frac{\pi}{2}\left(1+\frac{1}{4}k^2+\frac{9}{64}k^4+\frac{25}{256}k^6+\ldots \right),
$$
$$
E(k)=\frac{\pi}{2}\left(1-\frac{1}{4}k^2-\frac{3}{64}k^4-\frac{5}{256}k^6+\ldots \right),
$$
and
$$
\sqrt{k^4-k^2+1}=1-\frac{1}{2}k^2+\frac{9}{124}k^4+\frac{135}{720}k^6+\ldots,
$$
we deduce that,
$$r(k)=-80\pi^4k^4+O(k^6).$$
This shows, therefore, that $f_2$ is positive. 

By using similar arguments, we can also show the positivity of $f_1$ and $f_3$. In particular, for $k$ small,
$$
f_1(k)= \frac{80}{3}\pi^3-\frac{140}{3}\pi^3k^2+O(k^4)
$$
and
$$
f_3(k)= \frac{63}{128}\pi^3k^8+O(k^{10}).
$$
The proof of the lemma is thus completed.
\end{proof}

\begin{remark}
In the proof of Lemma \ref{lemma5.3}, we have used that $m_1(k)<0$ in order to ensure that we can replace $L^2$ in \eqref{t2} by $16\pi^2$ (recall that $L>4\pi$). Since $m_1(k)$ depends only on $K(k)$ and $E(k)$, one can prove that $m_1(k)<0$ following similar arguments as those above.
\end{remark}

Next, we turn attention to show how properties $(P_1)$-$(P_3)$ imply Theorem \ref{bbmmeantheorem}. The functional $F_{(c,A)}$ in \eqref{quantidadeconservadaF} is defined  in such a way that  \eqref{el} is its   Euler-Lagrange equation. Taking into account that $c$ and $A$ are functions of $k$, we now consider the functional $F_k:=F_{(c(k),A(k))}$, given by
\begin{equation}
F_k=E+(c(k)-1)Q+A(k)V.
\label{bbm4}
\end{equation}
In $H^1_{per}([0,L])$, let $\rho$ be the pseudo-metric  defined as 
$$
\rho(u,v):=\inf_{r\in\R} \| u-v(\cdot+r)\|_{H_{per}^1},
$$
Given any $\varepsilon>0$, $U_\varepsilon(\phi_k)$ denotes the $\varepsilon$-neighborhood of $\phi_k$ with respect to $\rho$, that is,
$$
U_\varepsilon(\phi_k)=\{u\in H^1_{per}([0,L]); \; \rho(u,\phi_k)<\varepsilon\}.
$$
We also introduce the manifold $\Sigma_k$  as
$$
\Sigma_k:=\{u\in H^1_{per}([0,L]);\,\,  M_k(u)=M_k(\phi_k)\}.
$$

Now, we state two classical lemmas. The proofs in our case are very close to the original ones.

\begin{lemma}\label{bbm14}
There exist $\varepsilon>0$ and a $C^1$ map  $\omega:U_\varepsilon(\phi_k) \to \mathbb{R}$, such that for all $u\in U_\varepsilon(\phi_k)$,
$$\Big( u( \cdot+\omega(u)),\phi_k'\Big)_{L^2_{per}}=0.
$$
\end{lemma}
\begin{proof}
The proof is based on an application of Implicit Function Theorem. See \cite[Lemma 4.1]{bss} or \cite[Lemma 7.7]{angulo4} for details.
\end{proof}

\begin{lemma}\label{bbm15}
Let
$$
\mathcal{A}=\left\{ \psi\in H^1_{per}([0,L]); ( \psi,M_k'(\phi_k))_{L^2_{per}}= ( \psi, \phi_k')_{L^2_{per}}=0\right\}.
$$
Under assumptions  $(P_0)$-$(P_3)$
 there exists $C>0$ such that 
$$\langle \mathcal{L}_k\psi,\psi\rangle \geq C \| \psi\|_{H^1_{per}}^2, \quad \psi \in \mathcal{A}.
$$
\end{lemma}
\begin{proof}
The proof is quite standard by now, so we omit the details. We refer the interested reader to \cite[Theorem 3.3]{Grillakis} or \cite[Lemma 7.8]{angulo4}. It should be noted that the assumption $d''(c)>0$ in such references must be replaced by $\Phi<0$.
\end{proof}

In the next lemma we prove that $\phi_k$ is a local minimum of the functional  $F_k$  restrict to the manifold $\Sigma_k$. It worth mentioning that its proof relies on the classical ideas with some changes in the spirit of  \cite[Lemma 4.6]{johnson1}.

\begin{lemma} \label{lemacoercividadebbm}
 Under the above assumptions, there exist $\varepsilon >0$  and a constant $C=C(\varepsilon)$ such that
\begin{equation}F_k(u)-F_k(\phi_k)\geq C\rho(u,\phi_k)^2,\label{bbm3}\end{equation}
for all  $u\in U_\varepsilon(\phi_k)$  satisfying  $M_k(u)=M_k(\phi_k)$.
\end{lemma}
\begin{proof}
Since $F_k$ is invariant by translations, we have $F_k(u)=F_k(u(\cdot+r))$, for all  $ r\in\mathbb{R}$.  Thus, it suffices to prove that
$$F_k(u(\cdot+\omega(u)))-F_k(\phi_k)\geq C \rho(u,\phi_k)^2,$$
where $\omega$ is given in Lemma \ref{bbm14}. 

By fixing $u\in U_\varepsilon(\phi_k)\cap\Sigma_k$ (with $\varepsilon$ as in Lemma \ref{bbm14}) and making use of Lemma \ref{bbm14}, it follows that there exists $C_1\in\mathbb{R}$ such that
$$v:=u(\cdot+\omega(u))-\phi_k=C_1M_k'(\phi_k)+y,$$
where $y\in \mathcal{T}_k = \{M_k'(\phi_k)\}^\perp \cap\{\phi_k'\}^\perp$. Since  $u$ belongs to $ U_\varepsilon(\phi_k)$, up to a translation  in $\phi_k$, we may assume that $v=u(\cdot+\omega(u))-\phi_k$ satisfies $\|v\|_{H^1_{per}}<\varepsilon$. 

Let us prove that $C_1=O(\|v\|^2)$. In fact, using the invariance by translation of  $M_k$, a Taylor expansion gives
\begin{equation}
M_k(u)=M_k(u(\cdot+\omega(u)))=M_k(\phi_k)+\langle M_k'(\phi_k),v\rangle +O(\|v\|^2).
\label{bbm5}
\end{equation}
On the other hand, since $y\in\mathcal{T}_k$, we have $\langle M_k'(\phi_k),y\rangle=0$ and
\begin{equation}\langle M_k'(\phi_k),v\rangle=\langle M_k'(\phi_k), C_1M_k'(\phi_k)+y\rangle=C_1\langle M_k'(\phi_k), M'_k(\phi_k)\rangle= C_1N,
\label{bbm6}
\end{equation}
where $N$ is a constant depending only on $k$. Therefore, since $M_k(u)=M_k(\phi_k)$, it follows from  (\ref{bbm5}) and (\ref{bbm6}) that
\begin{equation}
C_1=O(\|v\|^2).
\label{bbm7}
\end{equation}

Because $\phi_k$ is strictly positive and $F_k$ is smooth away from zero, a Taylor expansion  at   $u(\cdot +\omega(u))=\phi_k+v$ yields
$$F_k(u)=F_k(u(\cdot+\omega(u)))=F_k(\phi_k)+\langle F_k'(\phi_k),v\rangle+\frac{1}{2}\langle F_k''(\phi_k)v,v\rangle + o(\|v\|^2).$$
Since $F'_k(\phi_k)=0$ and $F_k''(\phi_k)=\mathcal{L}_k $, we have 
\begin{equation}
F_k(u)-F_k(\phi_k)=\frac{1}{2}\langle \mathcal{L}_k  v,v\rangle+o(\|v\|^2).
\label{bbm8}
\end{equation}

Note that the equality $v=C_1M_k'(\phi_k)+y$ provides
\begin{equation}\label{bbm8.1}
\langle \mathcal{L}_k  v,v\rangle  = C_1^2 \langle\mathcal{L}_k  M'_k(\phi_k),M'_k(\phi_k)\rangle+ 2 C_1\langle \mathcal{L}_k  M'_k(\phi_k), y\rangle +\langle\mathcal{L}_k  y,y\rangle.
\end{equation}
In view of \eqref{bbm7}, we obtain positive constants $C_2, C_3$, and $C_4$, depending only on $k$, such that
$$|C_1^2 \langle\mathcal{L}_k  M'_k(\phi_k),M'_k(\phi_k)\rangle|\leq C_2\|v\|^4$$
and 
\[
\begin{split}
\ds |2C_1\langle \mathcal{L}_k  M'_k(\phi_k), y\rangle|& \ds\leq 2 |C_1| \|\mathcal{L}_k  M'_k(\phi_k)\|\|y\|\\
&\ds \leq 2 |C_1| \|\mathcal{L}_k  M'_k(\phi_k)\|\Big(\|y+C_1M_k'(\phi_k)\|+\|C_1M_k'(\phi_k)\|\Big)\\
&\ds\leq C_3\|v\|^3+C_4\|v\|^4.\\
\end{split}
\]
This last two inequalities together with \eqref{bbm8.1} imply
\begin{equation}\langle \mathcal{L}_k  v,v\rangle=\langle\mathcal{L}_k  y,y\rangle+ o(\|v\|^2).
\label{bbm9}
\end{equation}
Therefore, combining (\ref{bbm8}) with (\ref{bbm9}), we obtain
$$F_k(u)-F_k(\phi_k)=\frac{1}{2}\langle \mathcal{L}_k  y,y\rangle+o(\|v\|^2).$$
By using that $y\in\mathcal{T}_k$, we have  $y\in\mathcal{A}$. Thus, Lemma \ref{bbm15} gives
$$\langle \mathcal{L}_k  y,y\rangle \geq C \| y\|_{H^1_{per}}^2.$$ 
So,
\begin{equation}
F_k(u)-F_k(\phi_k)\geq C\|y\|_{H^1_{per}}^2+o(\|v\|^2).
\label{bbm12}
\end{equation}
By using the definition of $v$ and \eqref{bbm7} it is easily seen that 
\begin{equation}
\|y\|_{H^1_{per}}^2\geq \|v\|_{H^1_{per}}^2+o(\|v\|_{H^1_{per}}^2),
\label{bbm13}
\end{equation}
provided $v$ is small enough (if necessary we can take a smaller $\varepsilon>0$).

Finally,  (\ref{bbm12}) and (\ref{bbm13}) combine to establish that
$$
F_k(u)-F_k(\phi_k)\geq C\|v\|_{H^1_{per}}^2+o(\|v\|_{H^1_{per}}^2),
$$
which, for $\varepsilon>0$ sufficient small, gives
$$
F_k(u)-F_k(\phi_k)\geq C(\varepsilon)\|v\|_{H^1_{per}}^2\geq C(\varepsilon) \rho(u,\phi_k)^2
$$
and completes the proof of the lemma.
\end{proof}

\begin{remark}
If we follow the same strategy as in \cite{johnson1}, we can also show that $F_k$ is coercive on the codimension two manifold
$$
\tilde{\Sigma}_k=\{u\in H^1_{per}([0,L]);\; Q(u)=Q(\phi_k),\; V(u)=V(\phi_k)\}.
$$
However, in this situation our Theorem \ref{bbmmeantheorem} would be restricted to perturbations in $\tilde{\Sigma}_k$. Since we only dispose of a one-parameter family of periodic waves, we do not know how to use a triangle-type inequality in order to prove Theorem \ref{bbmmeantheorem} for general perturbations in $H^1_{per}([0,L])$. Note that clearly, $\tilde{\Sigma}_k\subset \Sigma_k$.
\end{remark}

Having disposed of this preliminaries steps, we now turn  to prove our main theorem.

\begin{proof}[Proof of Theorem \ref{bbmmeantheorem}] The proof follows standard lines with some modifications. Assume by contradiction that  $\phi_k$ is  $H^1_{per}$-unstable. Then, for each $n\in\N$, we can choose an initial data  $w_n:=u_n(0)\in U_{\frac{1}{n}}(\phi_k)$ and $\varepsilon>0$ such that
$$\rho(w_n,\phi_k)\rightarrow 0, \quad\mbox{but}\quad \sup_{t\geq0}\rho(u_n(t),\phi_k)\geq \varepsilon ,$$
where $u_n(t)$ is the solution of  (\ref{bbm}) with initial data  $w_n$. Here, we choose $\varepsilon>0$ from  Lemma \ref{lemacoercividadebbm}.  By continuity in $t$, we can also choose the first time  $t_n>0$, such that
\begin{equation}
\rho(u_n(t_n),\phi_k)=\frac{\varepsilon}{2}.
\label{bbm16}
\end{equation}

The plan now is to obtain a contradiction with  (\ref{bbm16}). For that, let $f_n$ be the function defined as
\[
\begin{split}
f_n(\alpha)&=M_k(\alpha u_n(t_n))\\
&=\alpha^2\frac{\partial c}{\partial k}\int_0^L\Big(|u_n(t_n)|^2+|u_n'(t_n)|^2\Big)dx+\alpha\frac{\partial A}{\partial k}\int_0^Lu_n(t_n)dx=:\alpha^2a_n+\alpha b_n. 
\end{split}
\]
Since $f_n(0)=0$, $a_n>0$ and $M_k(\phi_k)>0$, it follows that there exists, for each   $n\in\mathbb{N}$, an $\alpha_n>0$ such that $f_n(\alpha_n)=M_k(\phi_k)$. The fact that   $M_k(\phi_k)>0$ can be proved in view of the explicit form of the quantities involved following the arguments in Lemma \ref{lemma5.3}. Since it demands too many calculations we omit the details.

As a consequence, we obtain a sequence of positive numbers  $(\alpha_n)_{n\in\mathbb{N}}$, such that
\begin{equation}
M_k(\alpha_nu_n(t_n))=M_k(\phi_k),\,\  n\in\mathbb{N}.
\label{bbm17}
\end{equation}
Therefore, the sequence  $(\alpha_nu_n(t_n))_{n\in\mathbb{N}}$ belongs to the manifold  $\Sigma_k$. This sequence will be the main ingredient to obtain a contradiction.

First, let us show that $(\alpha_n)_{n\in\mathbb{N}}$ admits at least one subsequence which converges to  $1$. In order to make the notation simpler, we set 
$$Q_k:=\frac{\partial c}{\partial k} Q\quad \mbox{and}\quad V_k:=\frac{\partial A}{\partial k} V.$$
The continuity of $Q$ and $V$ provides, as $n\to\infty$,
\begin{equation}\label{bbm18}
Q_k(w_n)\longrightarrow Q_k(\phi_k)=:a,\qquad
V_k(w_n)\longrightarrow V_k(\phi_k)=:b,
\end{equation}
and
\begin{equation}\label{18.1}
M_k(w_n)\longrightarrow M_k(\phi_k).
\end{equation}
Also, since $Q$ and $V$ are conserved quantities, 
\begin{equation}\begin{split}
\varrho& :=|\alpha_n^2Q_k(w_n)+\alpha_nV_k(w_n)-(Q_k(w_n)+V_k(w_n))|\\
&=|\alpha_n^2Q_k(u_n(t_n))+\alpha_nV_k(u_n(t_n))-(Q_k(w_n)+V_k(w_n))|\\
&=|Q_k(\alpha_nu_n(t_n))+V_k(\alpha_nu_n(t_n))-(Q_k(w_n)+V_k(w_n))|\\
&=|M_k(\alpha_nu_n(t_n))-M_k(w_n)|.
\end{split}
\label{bbm19}
\end{equation}
From \eqref{bbm17} and \eqref{18.1}, it follows that $\varrho\to0$, as $n\to\infty$.
Since
\begin{equation}\label{19.1}
\begin{split}
0&\leq |\alpha_n^2Q_k(w_n)+\alpha_nV_k(w_n)-(a+b)|\\
&\leq |\alpha_n^2Q_k(w_n)+\alpha_nV_k(w_n)-(Q_k(w_n)+V_k(w_n))|\\
&\quad + |(Q_k(w_n)+V_k(w_n))-(a+b)|\\
&\leq \varrho +|Q_k(w_n)-a| + |V_k(w_n)-b|
\end{split}
\end{equation}
and, in view of \eqref{bbm18},  the last line in \eqref{19.1} goes to 0, as $n\to\infty$, we deduce that
\begin{equation}
z_n:=\alpha_n^2Q_k(w_n)+\alpha_nV_k(w_n) \longrightarrow a+b, \quad \mbox{as} \;\; n\rightarrow \infty.
\label{bbm20}
\end{equation}

Next we claim that $(\alpha_n)_{n\in\mathbb{N}}$ is a bounded sequence. On the contrary, suppose  $(\alpha_n)_{n\in\mathbb{N}}$ is unbounded. Because $Q_k(w_n)>0$ and $Q_k(w_n)$ and $V_k(w_n)$ are bounded  we obtain, up to a subsequence,
$$ 
z_n=\alpha_n(\alpha_nQ_k(w_n)+V_k(w_n))\longrightarrow +\infty, \quad \mbox{as} \;\; n\rightarrow \infty,$$
which contradicts \eqref{bbm20}. Therefore, there exist a subsequence, which we still denote by  $(\alpha_n)_{n\in\mathbb{N}}$ and $\alpha_0\geq0$ such that
$$
\alpha_n\longrightarrow \alpha_0, \quad \mbox{as} \;\; n\rightarrow \infty.
$$
Taking the limit in $z_n$, it follows from \eqref{bbm18} and (\ref{bbm20}) that
$\alpha_0^2a+\alpha_0b=a+b,$
i.e.,
$$(1-\alpha_0)[(1+\alpha_0)a+b]=0.$$
This implies 
$$\alpha_0=1 \quad \mbox{or}\quad \alpha_0=-\frac{b+a}{a}=-\left(\frac{b}{a}+1\right).$$
The fact that $M_k(\phi_k)>0$ and $Q(\phi_k)>0$ combined with \eqref{deric} yield
$$
1+\frac{b}{a}=1+\frac{\frac{\partial A}{\partial k}V(\phi_k)}{\frac{\partial c}{\partial k}Q(\phi_k)}>0.
$$
Since $\alpha_0\geq0$, we therefore obtain  $\alpha_0=1$.

Now, we will use the sequence  $(\alpha_nu_n(t_n))_{n\in\mathbb{N}}$ and its properties, to prove the next two claims.

\vskip.3cm
\noindent {\bf{Claim 1:}} $\rho(u_n(t_n),\alpha_nu_n(t_n))\longrightarrow 0$, as $ n\rightarrow \infty.$

In fact, by definition,
\begin{equation}
\begin{split}
\ds\rho(u_n(t_n),\alpha_nu_n(t_n))&\ds=\inf_{r\in\mathbb{R}}\|u_n(\cdot,t_n)-\alpha_nu_n(\cdot+r,t_n)\|_{H^1_{per}}\\
&\leq \|u_n(t_n)-\alpha_nu_n(t_n)\|_{H^1_{per}}=|(1-\alpha_n)|\|u_n(t_n)\|_{H^1_{per}}.\\
\end{split}
\label{bbm21}
\end{equation}
On the other hand, since $\ds\rho(u_n(t_n),\phi_k)=\frac{\varepsilon}{2}$, there exists $r\in\mathbb{R}$ such that
$$
\|u_n(t_n)\|_{H^1_{per}}\leq \|u_n(t_n)-\phi_k(\cdot+r)\|_{H^1_{per}}+\|\phi_k(\cdot+r)\|_{H^1_{per}}<\varepsilon+\|\phi_k(\cdot+r)\|_{H^1_{per}},
$$
which is to say that $\left(\|u_n(t_n)\|_{H^1_{per}}\right)_{n\in\N}$ is uniformly bounded. Taking the limit in \eqref{bbm21}, as $n\to\infty$, and taking into account that $\alpha_n\longrightarrow 1$, we obtain the claim.

\vskip.3cm
\noindent {\bf{Claim 2:}} $\rho(\alpha_nu_n(t_n),\phi_k)\longrightarrow 0$, as $n\rightarrow\infty$.

In fact, from Claim 1 and (\ref{bbm16}), we see that
$$
\rho(\alpha_nu_n(t_n),\phi_k)\leq \rho(\alpha_nu_n(t_n),u_n(t_n))+\rho(u_n(t_n),\phi_k)<\frac{\varepsilon}{3}+\frac{\varepsilon}{2}=\frac{5\varepsilon}{6}<\varepsilon,
$$
for all $n$ large enough. This means that for $n$ large enough, $\alpha_nu_n(t_n)\in U_\varepsilon(\phi_k)$. Moreover, we have already seen in  (\ref{bbm17}) that $\alpha_nu_n(t_n)\in\Sigma_k$. These two facts tell us that   $\alpha_nu_n(t_n)\in\Sigma_k\cap U_{\varepsilon}(\phi_k)$. Consequently,   Lemma \ref{lemacoercividadebbm} implies
\[
\begin{split}
\rho(\alpha_nu_n(t_n),\phi_k)^2&\leq C|F_k(\alpha_nu_n(t_n))-F_k(\phi_k)|\\
&\leq C|F_k(\alpha_nu_n(t_n))-F_k(u_n(t_n))|+ C|F_k(u_n(t_n))-F_k(\phi_k)|\\
&=C|F_k(\alpha_nu_n(t_n))-F_k(u_n(t_n))|+ C|F_k(w_n)-F_k(\phi_k)|.
\end{split}
\]
Taking the limit, as $n\to\infty$, in the last inequality, the continuity of $F_k$ and the boundedness of $(\alpha_nu_n(t_n))_{n\in\mathbb{N}}$ in $H^1_{per}([0,L])$,
 yield Claim 2.
\vskip.3cm

Finally, Claims 1 and 2 combine to give
$$
\frac{\varepsilon}{2}=\rho(u_n(t_n),\phi_k)\leq \rho(u_n(t_n),\alpha_nu_n(t_n))+\rho(\alpha_nu_n(t_n),\phi_k) \longrightarrow 0,
$$
as $n\to\infty$, which is a contradiction. The proof of Theorem \ref{bbmmeantheorem} is thus established.
\end{proof}

\begin{remark}\label{remark5.8}
It should be noted that our family of periodic waves in Corollary \ref{lascor} is a zero-energy family, which means that the constant $B$ in \eqref{eqquad} was set to be zero.  As we note in Remark \ref{remB}, one can also obtain periodic solutions with $B\neq0$. However, our approach does not apply in this general situation and we do not know if such waves are stable or not.

Related to this, Johnson \cite{johnson2} studied the spectral and orbital stability of periodic waves for the generalized BBM equation   \eqref{genbbm} with $f(u)=u+g(u)$ for some $C^2$ function $g$. The author considered  the full family of periodic waves (depending on $c$, $A$, and $B$) and gave a criterion to establish the orbital stability of such waves depending on the sign of certain determinants, which encode some geometric information about the underlying manifold of periodic solutions.
\end{remark}

\section*{Appendix}
In this short Appendix, we collect the formulas for the functions introduced in the proof of Lemma \ref{lemma5.3}. The functions $p_i$ are given by
$$\begin{array}{l}
\ds p_1(k):=(2k^6-7k^4+10k^2-5)+\left(2k^4-6k^2+4\right)\sqrt{k^4-k^2+1},
\\ \\
\ds p_2(k):= (-4k^6+15k^4-21k^2+14)+(-4k^4+4k^2-4) \sqrt{k^4-k^2+1},
\\ \\
\ds p_3(k):= (-9k^4+9k^2-9),
\\ \\
\ds p_4(k):=(-k^8+4k^6-6k^4+5k^2
-2)+(-k^6-k^4+4k^2-2)\sqrt{k^4-k^2+1},
\\ \\
\ds p_5(k):= (2k^8-4k^6+6k^4-4k^2+2)+(2k^6-3k^4-3k^2+2)\sqrt{k^4-k^2+1},
\\ \\
\ds p_6(k):=\left(40k^{10}-276k^8+686k^6-878k^4+642k^2-214\right)\\
\hspace{1,4cm}+( 40k^8-175k^6+300k^4-277k^2+110)\sqrt{k^4-k^2+1},
\\ \\
\ds   p_7(k):=(-80k^{10}+494k^8-1256k^6+1684k^4-1270k^2+508)\\
\hspace{1,4cm}+(-80k^8+130k^6-270k^4+280k^2-140)\sqrt{k^4-k^2+1},
\\ \\
\ds p_8(k):=(-294k^8+588k^6-882k^4+588k^2-294)\\
\hspace{1,4cm}+(30k^6-45k^4-45k^2+30)\sqrt{k^4-k^2+1},
\\ \\
\ds p_9(k):= (10k^{10}-55k^8+125k^6-155k^4+105k^2-30)\\
\hspace{1,4cm} +(28k^8-98k^6+154k^4-126k^2+42)\sqrt{k^4-k^2+1},
\end{array}
$$

$$\begin{array}{l}
\ds p_{10}(k):=(30k^{10}-15k^8-90k^6+195k^4-180k^2+60)\\
\hspace{1,5cm} +(-42k^8+168k^6-252k^4+210k^2-84)\sqrt{k^4-k^2+1},
\end{array}$$
and
$$\begin{array}{l}
\ds p_{11}(k):= (-30k^{10}+75k^8-30k^6-30k^4+75k^2-30)\\
\hspace{1,5cm}+(42k^8-84k^6+126k^4-84k^2+42)\sqrt{k^4-k^2+1}.
\end{array}
$$
The functions $m_i$ are defined as
$$\begin{array}{l}
m_1(k):=\left(p_6(k)K^2(k)+p_7(k)E(k)K(k)+p_8(k)E^2(k)\right),
\\ \\
m_2(k):=\left(p_9(k)K^2(k)+p_{10}(k)E(k)K(k)+p_{11}(k)E^2(k)\right),
\\ \\
m_3(k):=\left(2k^6-3k^4-3k^2+2+(2k^4-2k^2+2)\sqrt{k^4-k^2+1}\right)\times \\
\hspace{1,6cm}\left(p_1(k)K^2(k) +p_2(k)E(k)K(k)+p_3(k)E^2(k)\right),
\end{array}$$
and
$$\begin{array}{l}
m_4(k):=\left(\left((k^4-k^2+1)+(k^2-2)\sqrt{k^4-k^2+1}\right)K(k)+3\sqrt{k^4-k^2+1}E(k)\right)\times
\\ 
\hspace{1,6cm}\left(p_4(k)K(k)+p_5(k)E(k)\right).
\end{array}
$$
\section*{Acknowledgements}

This work is part of the Ph.D Thesis of the first author, which was concluded at IMECC-UNICAMP, under the guidance of the second author. The first author acknowledges the financial support from Capes and CNPq. The second author is partially supported by CNPq and FAPESP. The authors gratefully acknowledge the referees for the constructive remarks which improve the presentation of the manuscript.


\begin{thebibliography}{100}

\bibitem{angulo7}
B. Alvarez and J. Angulo, Existence and stability of periodic travelling-wave solutions of the Benjamin equation,  \textit{Commun. Pure Appl. Anal.} 4 (2005), 367--388.

\bibitem{AN2} J. Angulo and F. Natali, Stability and instability of periodic
travelling waves solutions for the critical Korteweg-de Vries and
non-linear Schr\"odinger equations,  \textit{Phys. D} 238 (2009)
603--621.

\bibitem{angulo4}
J. Angulo, Nonlinear dispersive equations: Existence and stability of solitary and periodic travelling wave solutions, Math. Surveys Monogr.   156, American Mathematical Society, 2009.

\bibitem{angulo5}
J. Angulo and A. Pastor, Stability of periodic optical solitons for a nonlinear Schrödinger system, \textit{Proc. Roy. Soc. Edinburgh Sect. A} 139 (2009), 927--959.





\bibitem{angulo3} J. Angulo, Nonlinear
stability of periodic travelling wave solutions to the
Schrödinger and modified Korteweg-de Vries equations,   \textit{J. Differential Equations 235} (2007), 1--30.


\bibitem{acn}
J. Angulo, E. Cardoso and F. Natali, Stability properties of periodic traveling waves for the Intermediate Long Wave equation, arXiv:1503.04350.

\bibitem{abs} J. Angulo, C. Banquet, and M. Scialom, Stability for the modified and fourth-order Benjamin-Bona-Mahony equations,   \textit{Discrete Contin. Dyn. Syst.} 30 (2011), 851--871.


\bibitem{angulo1} J. Angulo, J. L. Bona, and M. Scialom, Stability of cnoidal waves,  \textit{Adv.  Differential Equations} 11 (2006), 1321--1374.

\bibitem{be}  T. B. Benjamin, The stability of solitary waves,  \textit{Proc. Roy. Soc. (London) Ser. A} 328 (1972), 153--183.

\bibitem{bbm} T. B. Benjamin, J. L. Bona and J. J. Mahony, Model equations for long waves in nonlinear dispersive systems,  \textit{Phil. Trans. Royal Soc. London, Ser. A} 272 (1972), 47--78.

\bibitem{bss}  J. L. Bona, P. E. Souganidis, and W. A. Strauss, Stability and instability of solitary waves of the Korteweg-de Vries type,  \textit{Proc. R. Soc. Lond. A} 411 (1987), 395--412.

\bibitem{bmr}  J. L. Bona, W. R. McKinney, and J. M. Restrepo, Stable and unstable solitary-wave solutions of the generalized regularized long-wave equation,  \textit{J. Nonlinear Sci.} 10 (2000), 603--638.

\bibitem{brosl}  L. J. F. Broer and F. W. Sluijter, Stable approximate equations for ion-acoustic waves,  \textit{Phys. Fluids} 20 (1977), 1458--1460.

\bibitem{brovati}  L. J. F. Broer, E. W. C. Van Groesen, and J. M. W. Timmers, Stable model equations for long water waves,  \textit{Appl. Sci. Res.} 32 (1976), 619--636.

\bibitem{bjk}  J. C. Bronski, M. A. Johnson, and T. Kapitula, An index theorem for the stability of periodic travelling waves of Korteweg-de Vries type,  \textit{Proc. Roy. Soc. Edinburgh Sect. A} 141 (2011), 1141--1173.


\bibitem{depas} T. P. de Andrade and A. Pastor, Nonlinear stability of periodic waves for modified Korteweg-de Vries and Gardner equations,  in preparation.


\bibitem{friedman}
P.F. Byrd and M.D. Friedman,  Handbook of elliptic integrals for engineers and scientists, 2nd ed., Springer, NY, 1971.

\bibitem{ea} M. S. P. Eastham, The Spectral Theory of Periodic Differential
Equations, Scottish Academic Press, Edinburgh, 1973.

\bibitem{fp} L. G. Farah and A. Pastor, On the periodic Schr\"odinger-Boussinesq system,  \textit{J. Math. Anal. Appl.} 368 (2010),  330--349.

\bibitem{gill} T. S. Gill, C. Bedi, and N. S. Saini, Higher order nonlinear effects on wave structures in a four-component dusty plasma with nonisothermal electrons,  \textit{Phys. Plasmas} 18 (2011),  043701.

\bibitem{Grillakis} M. Grillakis, M. Shatah and W. Strauss, Stability theory
of solitary waves in the presence of symmetry I,  \textit{J. Functional
Anal.} 74 (1987), 160--197.

\bibitem{hik} S. Hakkaev, I. D. Iliev, and K. Kirchev,  Stability of periodic traveling waves for complex modified Korteweg-de Vries equation, \textit{J. Differential Equations} 248 (2010), 2608--2627.

\bibitem{jack}
J. K. Hale,  Ordinary Differential Equation, Dover Publications, Revised Edition, New York, 1980.


\bibitem{ince} E. L. Ince,  The periodic Lam{\'e} functions,  \textit{Proc. Roy. Soc. Edinburgh} 60 (1940), 47--63.

\bibitem{johnson1} M. A. Johnson, Nonlinear stability of periodic traveling wave solutions of the generalized Korteweg-de Vries equation,  \textit{SIAM J. Math. Anal.} 41 (2009), 1921--1947.

\bibitem{johnson2} M. A. Johnson, On the stability of periodic solutions of the generalized Benjamin-Bona-Mahony equation,  \textit{Physica D} 239 (2010), 1892--1908.

\bibitem{lin} Z. Lin, Instability of nonlinear dispersive solitary waves,  \textit{J. Funct. Anal.} 255   (2008), 1191--1224.

\bibitem{natali2} F. Natali and A. Neves, Orbital stability of periodic waves,  \textit{IMA J. Appl. Math.} 79   (2013), 1--19.

\bibitem{NP1} F. Natali and A. Pastor,
Stability properties of periodic standing waves for the
Klein-Gordon-Schrödinger system,\textit{Commun. Pure Appl. Anal.} 9 (2010), 413--430.

\bibitem{natali-pastor} F. Natali and A. Pastor, Stability and instability of periodic standing wave solutions for some Klein-Gordon equations, \textit{J. Math. Anal. Appl.} 347 (2008), 428--441.


\bibitem{Neves1}
A. Neves,  Floquet's Theorem and stability of periodic solitary waves, \textit{J. Dynam. Differential Equations} 21 (2009), 555--565.

\bibitem{Neves2}
A. Neves, Isoinertial family of operators and convergence of KdV cnoidal waves to solitons, \textit{J. Differential Equations} 244 (2008), 875--886.

\bibitem{sch}
H. Schamel, Stationary solitary, Snoidal and Sinusoidal ion acoustic waves, \textit{Plasma Physics} 14 (1972), 905--924.

\bibitem{sch1}
H. Schamel, A modified Korteweg-de Vries equation for ion acoustic waves due to resonant electrons, \textit{J. Plasma Physics} 9 (1973), 377--387.

\bibitem{ss}
P.E. Souganidis and W.A. Strauss, Instability of a class of dispersive solitary waves, \textit{Proc. Roy. Soc. Edinburgh Sect. A} 114 (1990), 195--212.

\bibitem{tran} M. Q. Tran, Ion acoustic solitons in a plasma: A review of their
experimental properties and related theories,  \textit{Phys. Scr.}  20 (1979),  317--327.

\bibitem{weinstein1} M. I. Weinstein, Lyapunov stability of ground states of nonlinear dispersive evolution equations,  \textit{Comm. Pure Appl. Math.}  39 (1986),  51--67.

\bibitem{weinstein2} M. I. Weinstein, Existence and dynamic stability of solitary wave solutions of equations arising in long wave propagation, \textit{Comm. Partial Differential Equations} 12 (1987), 1133--1173.

 \end{thebibliography}
\end{document}